\documentclass{birkjour}
\usepackage[colorlinks=true%,pagebackref
]{hyperref} %,pagebackref
\usepackage{setspace}

\usepackage{amssymb}
\usepackage{amsthm}
\usepackage{amsmath}%\usepackage{theorem}

\usepackage{psfig}
\usepackage{hyperref}

\usepackage{xcolor,enumerate,array}
\usepackage[toc,page]{appendix}
\usepackage[color,all]{xy}

\input xy
\xyoption{all}

\newtheorem{thm}{Theorem}[section]

\newtheorem{prop}[thm]{Proposition}
\theoremstyle{definition}

\theoremstyle{remark}

\numberwithin{equation}{section}

\numberwithin{equation}{section}
\numberwithin{thm}{section}

\numberwithin{equation}{section}
\numberwithin{thm}{section}
\begin{document}

%\begin{frontmatter}



\title[]
{Multiple Riemann wave solutions of the general form of quasilinear hyperbolic systems}

%----------Author 1
\author[A.M. Grundland]{A.M. Grundland}
\address{Centre de Recherches Math\'ematiques, Universit\'e de Montr\'eal,\\
	Succ. Centre-Ville, CP 6128, Montr\'eal (QC) H3C 3J7, Canada \\
	Department of Mathematics and Computer Science, Universit\'e du Qu\'ebec, 
	CP 500, Trois-Rivi\`eres (QC) G9A 5H7, Canada.}

\email{grundlan@crm.umontreal.ca}

%\thanks{This work was completed with the support of our
	%	\TeX-pert.}
%----------Author 2

\author{J. de Lucas}
\address{Department of Mathematical Methods in Physics, University of Warsaw, \\ ul. Pasteura 5, 02-093, Warszawa, Poland.}
\email{javier.de.lucas@fuw.edu.pl}
\subjclass{{35Q53} (primary); 35A30, 35Q58, 53A05 (secondary)}

\keywords{$k$-wave solution, Riemann invariant, conditional symmetry, Lie point symmetry. }

\date{}

%\title{On the geometry of the Clairin theory of conditional symmetries\\ for higher-order systems of PDEs with applications}


%\author{A.M. Grundland}

\begin{abstract}
	The objective of this paper is to construct geometrically Riemann $k$-wave solutions of the general form of first-order quasilinear hyperbolic systems of partial differential equations. To this end, we adapt and combine elements of two approaches to the construction of Riemann $k$-waves, namely the symmetry reduction method and the generalized method of characteristics. We formulate a geometrical setting for the general form of the $k$-wave problem and discuss in detail the  conditions for the existence of $k$-wave solutions. An auxiliary result concerning the Frobenius theorem is established. We use it to obtain formulae describing the $k$-wave solutions in closed form. Our theoretical considerations are illustrated by examples of hydrodynamic type systems including the Brownian motion equation.
	% This work presents a geometrical formulation of first-order quasilinear hyperbolic systems with coefficients depending on independent and dependent variables, called hydrodynamical systems, including the estimation of the number of degrees of freedom of their general analytic solutions written in terms of their admissible Riemann invariants. A generalisation of the Riemann invariants method has been formulated for these systems, including the conditions for the existence of Riemann $k$-wave solutions. These solutions can be interpreted as superpositions of $k$ simple waves. %New techniques for the reduction of these systems to express them in terms of Riemann invariants are presented. In particular, a geometric theory of $k$-rank waves  is presented. % New superposition rules leading to nonlinear superpositions of Riemann $k$-waves are formulated.
	%Our theoretical considerations are illustrated by examples of physical and mathematical interest.% hydrodynamic type systems.
\end{abstract}
\maketitle
%\begin{flushright}
%{\small This work is dedicated to the memory of professor Pavel Winternitz\\ (Centre the Recherches Math\'ematiques, Universit\'e de Montr\'eal).}
%\end{flushright}
\section{Introduction}\label{Introduction}\setcounter{equation}{0}

Riemann waves represent a very important class of solutions of first-order nonlinear systems of partial differential equations (PDEs). They are ubiquitous in  mathematical physics since they are admitted in all multidimensional hyperbolic systems and constitute their elementary solutions. Their characteristic feature is that, in most cases, they are expressible in implicit form. 

In this section we consider the possibility of generating the idea of Riemann wave solutions of a hyperbolic system of partial differential equations to what may be considered a nonlinear superposition of Riemann waves. This fact will be explained in detail in Sections 5 and 6.  All functions are assumed to be smooth. We now discuss a homogeneous hyperbolic quasilinear first-order system   of $q$ equations in the  dependent (properly determined) variables $u=(u^1,\ldots,u^q)$ and  independent variables $x=(x^1,\ldots,x^p)$ of the form
\begin{equation}
	\label{Eq:HydEq}
	\sum_{i=1}^pA^i(u)u_i=0, \qquad u_j=\frac{\partial u}{\partial x^j},\qquad j=1,\ldots,p,
\end{equation}
where $A^1,\ldots,A^p$ are $q\times q$ matrices whose coefficients are  functions depending on $u$. All considerations here are local, so it suffices to search for solutions defined on a neighbourhood of $x=0$. A {\it Riemann wave} solution for (\ref{Eq:HydEq}) is defined locally on a neighbourhood of $x=0$ as a solution of the equation
\begin{equation}
	\label{eq:udep}
	u=f(r(x,u)),
\end{equation}
for certain functions $f:\mathbb{R}\rightarrow \mathbb{R}^q$ and  $r:\mathbb{R}^{p+q}\rightarrow \mathbb{R}$ given by
$$ 
r(x,u)=\sum_{i=1}^p\lambda_i(u)x^i,
$$
where the function $r$ is called a {\it Riemann invariant} (RI) associated with the {\it wave vector}  $\lambda=(\lambda_1,\ldots,\lambda_p)$. Namely, it is assumed that
$$
\ker\left(\sum_{i=1}^p\lambda_i(u)A^i(u)\right)\neq 0.
$$
Equation (\ref{eq:udep}) defines a unique function $u(x)$ on a neighbourhood of $x=0$ for any function $f:\mathbb{R}\rightarrow \mathbb{R}^q$ and
$$
\frac{\partial u^\alpha}{\partial x^i}(x)=\phi(x)^{-1}\lambda_i(u(x))\frac{df^{\alpha}}{dr}(r(x,u(x))),\quad i=1,\ldots,p,\quad \alpha=1,\ldots,q,
$$
where 
$$
\phi(x)=1-\sum_{\alpha=1}^q\frac{\partial r}{\partial u^\alpha}(x,u(x))\frac{df^{\alpha}}{dr}(r(x,u(x))).%,\qquad f^{\alpha}\,\!':=\frac{df^\alpha}{dr}.
$$
The solutions (\ref{eq:udep}) have rank at most equal to one. They are the building blocks for constructing more general types of solutions that describe superpositions of many simple Riemann waves, which are much more interesting from the physical point of view. Until now, the only way to approach  this task was through the method of characteristics or via group theory. The method of characteristics relies on treating RIs as new dependent variables (which remain constant along appropriate characteristic curves of the basic system), which leads to the reduction of the dimension of the problem. The most important theoretical results obtained with the use of the method of characteristics (see e.g. \cite{CH62,HK89,Je76,JK84,Pe85,RJ83} and \cite{Wh74}  for a review%HK89,Pe85,RJ83,
) include the finding of some necessary and sufficient conditions for the existence of  so-called {\it Riemann $k$-waves} in multidimensional systems. It was shown in \cite{Ca53,Pe71,Pe85}, through the use of involutive differential equations in the Cartan sense, that these solutions of (\ref{Eq:HydEq}) depend on $k$ arbitrary functions of one variable. 

Certain criteria for determining the elastic or nonelastic character of the superposition between Riemann waves, particularly useful in physical applications, were also found \cite{Je76,Pe85,RJ83}. In applications to fluid dynamics and field theory, many new and interesting results were obtained \cite{Bo65,DG96,FK04,GL14,RJ83,Tr62}. The method of characteristics, like all other techniques for  solving PDEs, has its limitations. These limitations motivated the search for the means of constructing larger classes of multiple wave solutions expressible in terms of Riemann invariants, i.e.
\begin{equation}
	\label{Eq:Sol}
	u=f\left(r^{1}(x,u),\ldots,r^{k}(x,u)\right),
\end{equation}
where $f:\mathbb{R}^k\rightarrow \mathbb{R}$ and
$$
r^{a}(x,u)=\sum_{i=1}^p\lambda^{(a)}_i(u)x^i,\qquad 1\leq a\leq k.
$$
The natural way to do this was to look at these solutions from the point of view of the group invariance properties of the system (\ref{Eq:HydEq}).  The feasibility and clear advantages of such an approach were demonstrated for certain fluid dynamics equations \cite{DG96,GT96}.

The objective of this paper is to develop a systematic way of constructing multiple Riemann waves by making use of both approaches, namely the method of characteristics and the group analysis of differential equations. The basic feature of Riemann $k$-waves is that they remain constant on $(p-k)$- dimensional hyperplanes, where the set of linearly independent wave covectors $\lambda^{(1)},\ldots,\lambda^{(k)}$ entering the superposition vanish. In the context of group theory, this means that the graph  of a solution, $\Gamma=\{(x,u(x)):x\in \mathbb{R}^p\}$, is invariant under all vector fields on $\mathbb{R}^{p+q}$ of the form
\begin{equation}
	\label{eq:Vec}
	X_a=\sum_{i=1}^p\eta^i_a(u)\partial_{x^i},\qquad a=1,\ldots,k,
\end{equation}
that satisfy the orthogonality conditions
$$
\sum_{i=1}^p\eta^i_a(u)\lambda^{(a)}_i(u)=0,\qquad \forall a=1,\ldots,k.
$$
Then, $u(x)$ is a solution of (\ref{Eq:HydEq}) for some function $f$ because the functions $\{r^1,\ldots,r^k,u^1,\ldots,u^q\}$ constitute a complete set of invariants of the Abelian algebra $L$ of the vector fields (\ref{eq:Vec}). The implicit form of these solutions (\ref{Eq:Sol}) leads to some limitation in the application of the classical symmetry reduction method to this case. To overcome these limitations, we rectify the set of vector fields $X_1,\ldots,X_k$ by a change of independent and dependent  variables in $\mathbb{R}^{p+q}$, choosing RIs as the new independent variables. The initial system (\ref{Eq:HydEq}) expressed in these new coordinates, complemented by the invariance conditions for the rectified vector fields $X_1,\ldots,X_k$, form an overdetermined quasilinear system. Thus, the solutions of this system are invariant under the Abelian group corresponding to $L$. These vector fields $X_1,\ldots,X_k$ are conditional symmetries of the initial system (\ref{Eq:HydEq}).

In this paper we explore several different aspects of the $k$-wave problem. We investigate certain general properties of the general form of first-order quasilinear hyperbolic systems and propose several new approaches to the construction of its solutions. We do not obtain all solutions of this system of PDES, although the analysis of the obtained results can provide some interesting physical information concerning nonlinear superpositions of Riemann waves. 

The paper is organised as follows. Section 2 describes the basic theory of Riemann invariants. In Section 3, the relation  between inhomogeneous and the corresponding homogeneous first-order systems of PDEs in their general form is studied. Section 4 contains simple examples to illustrate the construction introduced in Section 3. Section 5 contains a geometric formulation of the Riemann $k$-wave problem, which represents a generalisation of the approach introduced in Section 2. In Section 6 a new technique for generating $k$-wave solutions   is proposed  and conditions for their existence are found. Obtaining these results required certain modifications of the Frobenius theorem, which are presented in the Appendix. In Section 7 we analyse the existence conditions for $k$-wave solutions and provide their geometric interpretation. Section 8 contains two examples of the application of our approach to hydrodynamic-type systems including the construction of Riemann double waves for the Brownian motion equation.

\section{Riemann invariants}\label{Se:RI}
This section reviews the Riemann invariants method for hyperbolic quasi-linear systems of first-order differential equations, also called {\it hydrodynamic equations}. Our approach provides some new elements of the geometric description of the method, which illuminate the proposed procedure for constructing multi-wave solutions.

Let us discuss a properly determined homogeneous quasi-linear hyperbolic system   of first-order partial differential equations (PDEs) of the form
\begin{equation}\label{Eq:Sys}
	\sum_{i=1}^p\sum_{\beta=1}^qA^{i\alpha}_\beta (u)u^\beta_i=0,\qquad \alpha=1,\ldots,q,
\end{equation}
in $p$ independent variables, $x^1,\ldots,x^p$, and $q$ dependent variables $u^1,\ldots,u^q$. Each variable $u^\beta_i$ stands for the derivative of $u^\beta$ with respect to $x^i$. All following considerations are local. Due to the invariance of (\ref{Eq:Sys}) under translations of the independent variables,  the study of (\ref{Eq:Sys}) can be reduced to the analysis of what happens close to the origin $x=0$. If we define $u_i=(u^1_i,\ldots,u_i^q)$, then (\ref{Eq:Sys}) can be rewritten as (\ref{Eq:HydEq}), where each $A^i$, for $i=1,\ldots,p$, is  a $q\times q$ matrix whose entries are functions that depend only on the dependent variables. Frequently, one writes $x^1,\ldots,x^n,t=x^{n+1}$ with $n+1=p$. Then, (\ref{Eq:Sys}) can be rewritten as 
\begin{equation}\label{pEF}
	\sum_{i=1}^nA^i(u)u_i+A^{n+1}(u)u_t=0.
\end{equation}
If $A^{n+1}={\rm Id}_{q}$, where ${\rm Id}_q$ is a $q\times q$ identity matrix,  then (\ref{pEF}) takes the so-called {\it evolutionary form}
\begin{equation}\label{Eq:bezx}
	u_t=-\sum_{i=1}^nA^i(u)u_i.
\end{equation}

A {\it wave covector} is a non-zero vector function $\lambda:\mathbb{R}^q\rightarrow \mathbb{R}^p\,\!^*$  such that 
$$
\lambda(u)=(\lambda_1(u),\ldots,\lambda_p(u)),\qquad \ker \left(\sum_{i=1}^p\lambda_i(u)A^i(u)\right)\neq 0.
$$
The function 
\begin{equation}\label{RI}
	r:(x,u)\in \mathbb{R}^p\times \mathbb{R}^q\mapsto \sum_{i=1}^p\lambda_i(u)x^i \in \mathbb{R},
\end{equation} where $x=(x^1,\ldots,x^p)\in \mathbb{R}^p$ and $u=(u^1,\ldots,u^q)\in \mathbb{R}^q$, 
is called the {\it Riemann invariant} associated with the covector $\lambda$. 

Consider a function
$
f:\mathbb{R}\rightarrow \mathbb{R}^q.
$
Let us first analyse when the condition $u=f(r(x,u))$ determines a solution $u=u(x)$ of (\ref{Eq:Sys}). The latter holds if and only if $g:(x,u)\in \mathbb{R}^p\times \mathbb{R}^q\mapsto f(r(x,u))-u\in \mathbb{R}^q$ satisfies  $g(x,u(x))=0$. The mapping $g$ is regular at every  point with $x=0$ since $g(0,u)=f(0)-u$ and the tangent mapping at 0, namely $T_0g$, in the basis $\{\partial_{ x_1},\ldots, \partial_{x_p}, \partial_{u_1},\ldots,\partial_{u_q}\}$, is given by 
$$
[T_{(x,u)}g]=\left[\frac{df^\alpha}{dr}\lambda_i(u)\bigg| D^\alpha_\beta\right],\qquad 
D_\beta^\alpha=\sum_{i=1}^p\frac{df^\alpha}{dr}\frac{\partial \lambda_i}{\partial u^\beta}(u)x^i-\delta^\alpha_\beta
$$ 
and 
$$
[T_{(0,u)}g]=\left[\frac{df^\alpha}{dr}\lambda_i(u)\bigg| -\delta^\alpha_\beta\right].
$$
The differentiability of $g$ ensures that $D^\alpha_\beta$ is regular in an open subset containing the points $(0,u)$ where $u\in \mathbb{R}^q$. By the implicit function theorem, $u=f(r(x,u))$  determines $u=u(x)$ in an open subset of $\mathbb{R}^p$ containing $x=0$. In particular, $u(0)=f(0)$. Moreover,
\begin{equation}\label{Eq:partial}
	\frac{\partial u^\alpha}{\partial x^i}(x)=\frac{df^\alpha}{dr}(r(x,u(x)))\left(\sum_{\beta=1}^q\sum_{j=1}^px^j\frac{\partial u^\beta}{\partial x^i}(x)\frac{\partial \lambda_j}{\partial u^\beta}(u(x))+\lambda_i(u(x))\right)
\end{equation}
amounts to 
$$ \sum_{\beta=1}^q\!\left[\delta^\alpha_\beta\!-\!\sum_{j=1}^p\frac{df^\alpha}{dr}(r(x,u(x))\frac{\partial \lambda_j}{\partial u_\beta}(u(x))x^j\right]\!\frac{\partial u^\beta}{\partial x^i}(x)\!=\!\frac{df^\alpha}{dr}(r(x,u(x)))\lambda_i(u(x)).
$$
The $q\times q$ matrix $\Psi$, whose entries depend on dependent and independent variables, is defined by
$$
\Psi^\alpha_\beta(x,u)=\delta^\alpha_\beta-\sum_{j=1}^p\frac{df^\alpha}{dr}(r(x,u))\frac{\partial \lambda_j}{\partial u^\beta}(u)x^j\Rightarrow \Psi^\alpha_\beta(0,u)=\delta^\alpha_\beta,\,\, \alpha,\beta=1,\ldots,q.
$$
Thus, $\Psi(x,u)$ is an invertible matrix in an open subset of each point $(0,u)$. Then,
\begin{equation}\label{eq1}
	\frac{\partial u^\alpha}{\partial x^i}(x)=\sum_{\beta=1}^q(\Psi^{-1})^\alpha_\beta(x,u(x)) \frac{df^\beta}{dr}(r(x,u(x)))\lambda_i(u(x))
\end{equation}
and $\sum_{\beta=1}^q\Psi^\alpha_\beta(x,u)\frac{df^\beta}{dr}(r(x,u))=I$ can be rewritten as 
\begin{multline*}
	I=\frac{df^\alpha}{d r}(r(x,u))-\sum_{\beta=1}^q\sum_{j=1}^p\frac{df^\alpha}{dr}(r(x,u))\frac{\partial\lambda_j}{\partial u^\beta}(u)x^j\frac{df^\beta}{d r}(r(x,u))\\=\frac{df^\alpha}{dr}(r(x,u))-\sum_{\beta=1}^q\frac{df^\alpha}{dr}(r(x,u))\frac{\partial r}{\partial u^\beta}(u)\frac{df^\beta}{dr}(r(x,u))\\=\left(1-\sum_{\beta=1}^q\frac{\partial r}{\partial u^\beta}(x,u)\frac{df^\beta}{dr}(r(x,u))\right)\frac{df^\alpha}{dr}(r(x,u)).
\end{multline*}
The latter yields
$$
\sum_{\beta=1}^q\Psi^\alpha_\beta(x,u)\frac{df^\beta}{dr}(r(x,u))=\left(1-\sum_{\beta=1}^q\frac{\partial r}{\partial u^\beta}(x,u)\frac{df^\beta}{dr}(r(x,u))\right)\frac{df^\alpha}{dr}(r(x,u)).
$$
Let us define 
$$\Phi(x,u)=1-\sum_{\beta=1}^q\frac{\partial r}{\partial u^\beta}(x,u)\frac{df^\beta}{dr}(r(x,u)).
$$
Hence,
$$
\sum_{\alpha,\beta=1}^q(\Psi^{-1})^\pi_\alpha(x,u)\Psi^\alpha_\beta(x,u)\frac{df^\beta}{dr}(r(x,u))=\sum_{\alpha=1}^q(\Psi^{-1})^\pi_\alpha(x,u)\Phi(x,u)\frac{df^\alpha}{dr}(r(x,u))$$
and 
$$
\frac{df^\pi}{dr}(r(x,u))=\sum_{\alpha=1}^q(\Psi^{-1})^\pi_\alpha(x,u)\Phi(x,u)\frac{df^\alpha}{dr}(r(x,u)).
$$
Using the previous equality in (\ref{eq1}), we obtain that
\begin{equation}\label{PDEu}
	\frac{\partial u^\beta}{\partial x^i}(x)=\sum_{\alpha=1}^q(\Psi^{-1})^\beta_\alpha(x,u)\frac{df^\alpha}{dr}(x,u)\lambda_i(u(x))=\frac{1}{\Phi(x,u)}\frac{df^\beta}{dr}(x,u)\lambda_i(u(x)).
\end{equation}
Let us now consider a function $\eta:u\in \mathbb{R}^q\mapsto (\eta^1(u),\ldots,\eta^p(u))\in \mathbb{R}^p$ such that $$\langle \lambda(u),\eta(u)\rangle=\sum_{i=1}^p\eta^i(u)\lambda_i(u)=0.
$$
A solution $u(x)$ of (\ref{PDEu}) and the function $\eta$ give rise to a vector field on $\mathbb{R}^{p}$ of the form
$$
X_{\eta,u(x)}(x)=\sum_{i=1}^p\eta^i(u(x))\frac{\partial}{\partial x^i}.
$$
In view of (\ref{PDEu}), we obtain
\begin{multline*}
	(X_{\eta,u(x)} u^\alpha)(x)=\sum_{i=1}^p\eta^i(u(x))\frac{\partial u^\alpha}{\partial x^i}(x)\\=\sum_{i=1}^p\eta^i(u(x))\Phi^{-1}(x,u(x))\frac{df^\beta}{dr}(r(x,u(x)))\lambda_i(u(x))=0. 
\end{multline*}
Hence, the coordinates $u^\alpha(x)$ of any particular solution $u(x)$  of the  hydrodynamic equations (\ref{Eq:HydEq}) are constants of the motion of $X_{\eta,u(x)}$ and therefore $u(x)$ is constant along its integral curves  (\ref{Eq:Sys}).

\section{A general inhomogeneous first-order system}\label{sec:kGen}

Let us study a more general case of (properly determined) hydrodynamic equations (\ref{Eq:Sys}) when an additional affine term appears explicitly and the matrix functions $A^i$ depend on the dependent variables $u^1,\ldots,u^q$, and also on the independent variables $x^1,\ldots,x^p$, namely
\begin{equation}\label{Eq:InPDE}
	\sum_{i=1}^p\sum_{\beta=1}^qA^{i\alpha}_\beta(x,u)u^\beta_i=b^\alpha(x,u),\qquad \alpha=1,\ldots,q,
\end{equation}
for certain functions $b^1,\ldots,b^q:\mathbb{R}^{p+q}\rightarrow \mathbb{R}$, 
where $x=(x^1,\ldots,x^p)\in \mathbb{R}^p$ and $u=(u^1,\ldots,u^q)\in \mathbb{R}^q$.
%As in the previous section, we set $u_i=(u_i^1,\ldots,u_i^q)$ and $b=(b^1,\ldots,b^q)$. Then, the (\ref{Eq:InPDE}) becomes
%$$
%\sum_{i=1}^pA^iu_i=b.
%$$
%In other words,
%$$
%A^0u_t+\sum_{i=1}^pA^iu_i=b
%$$
%and if $A^0={\rm Id}_q$, then
%$$
%u_t+\sum_{i=1}^pA^iu_i=b.
%$$

Let us show that an appropriate change of variables turns the inhomogeneous system (\ref{Eq:InPDE}) with $b\neq 0$ into a homogeneous one. Assume, without loss of generality, that $b^1\neq 0$ on an open subset $U\subset \mathbb{R}^{p+q}$. Then, there always exists an invertible $q\times q$ matrix $M$ whose coefficients are functions on $U$ such that $Mb=(1,0,\ldots,0)^T\in \mathbb{R}^q$. For instance, let us assume
\begin{equation}\label{Eq:M}
	M(x,u)=\frac{1}{b^1(x,u)}\left[\begin{array}{ccccc}
		1&0&0&\ldots&0\\
		-b^2(x,u)&b^1(x,u)&0&\ldots&0\\
		-b^3(x,u)&0&b^1(x,u)&\ldots&0\\
		\ldots&\ldots&\ldots&\ldots&0\\
		-b^n(x,u)&0&0&\ldots&b^1(x,u)\\
	\end{array}\right].
\end{equation}
Then,
\begin{equation}\label{Eq:InPDE2}
	\sum_{i=1}^p\sum_{\beta=1}^q\mathcal{A}^{i\alpha}_\beta(x,u)u^\beta_i=\delta_1^\alpha,\qquad \alpha=1,\ldots,q,
\end{equation}
where we set $\mathcal{A}^i(x,u)=M(x,u)A^i(x,u)$ and $\delta_1^\alpha$ is the Kronecker delta $\delta^\alpha_\beta$ with $\beta=1$. Define also
\begin{equation}\label{Eq:ChangeVar}
	\tilde{u}=u-x^{p+1}(1,0,\ldots,0)^T\in \mathbb{R}^q.
\end{equation}
Then, a particular solution $u(x^1,\ldots,x^p)$ of (\ref{Eq:InPDE}) induces a new function 
$$
\tilde{u}(x^1,\ldots,x^{p+1})=u(x^1,\ldots,x^p)-x^{p+1}(1,0,\ldots,0)^T\in \mathbb{R}^q
$$
and therefore
$$
\frac{\partial\tilde  u}{\partial x^i}(x^1,\ldots,x^{p+1})\!=\!\frac{\partial u}{\partial x^i}(x^1,\ldots,x^p),\,\, \frac{\partial \tilde u}{\partial x^{p+1}}(x^1,\ldots,x^{p+1})\!=\!(-1,0,\ldots,0)^T,
$$
for $i=1,\ldots,p$. Let us also define the $q\times q$ matrices, with $i=1,\ldots,p$ and coefficients given by functions on $\mathbb{R}^{p+q}$, of the form
$$
\tilde{A}^i(x,\tilde{u})=\mathcal{A}^i(x,\tilde{u}+x^{p+1}(1,0,\ldots,0) ^T),\,\, i=1,\ldots,p,\,\, \forall \tilde{u}\in \mathbb{R}^q, \,\, \forall x^{p+1}\in \mathbb{R},
$$
and let $\tilde{A}^{p+1}$ be the $q\times q$ identity matrix. Then, under the change of variables (\ref{Eq:ChangeVar}), the first-order inhomogeneous system (\ref{Eq:InPDE2}) becomes a quasilinear homogeneous first-order system of PDEs in $p+1$ independent variables
\begin{equation}\label{Eq:InPDE0}
	\sum_{i=1}^{p+1}\sum_{\beta=1}^q\tilde A^{i\alpha}_\beta(x,\tilde{u})\tilde{u}^\beta_i=0,\qquad \alpha=1,\ldots,q.
\end{equation}
%We hereafter assume then that we are dealing with this latter case.%$$
%\sum_{j=1}^{p+1}\tilde{A}^j      (x,\tilde{u})\frac{\partial \tilde{u}}{\partial x^j}=0.
%$$ 
Note that the above method allows us to recover solutions of the inhomogeneous system (\ref{Eq:InPDE}) as a certain subset of the solutions of a homogeneous equation (\ref{Eq:InPDE0}). However, a solution of (\ref{Eq:InPDE0}) does not need to give rise to a solution of the inhomogeneous system (\ref{Eq:InPDE}) since the expression
$$
\widetilde{u}(x^1,\ldots,x^{p+1})+x^{p+1}(1,0,\ldots,0)^T
$$
need not be a function depending only on $x^1,\ldots,x^p$.

The advantage of our previous approach is that homogeneous systems of first-order differential equations are often easier to solve than the corresponding inhomogeneous ones. 
As a consequence of previous considerations, we will discuss, in what follows, properly determined homogeneous quasi-linear first-order systems of PDEs. It is frequently convenient to denote the independent variables by $x^{1},x^2,\ldots,x^n, t=x^p$, where $n+1=p$, so that our hydrodynamic system, after reduction to the homogeneous form (\ref{Eq:InPDE0}), can be written in the evolutionary form 
$$
\tilde u_t+\sum_{i=1}^n\tilde{A}^i(x,\tilde u)\tilde u_i=0.
$$
In what follows, we suppress the tilde above the $u$'s and consider homogeneous systems in the dependent variables without tilde, namely \eqref{Eq:bezx}, 
%\begin{equation}\label{Eq:Sup}
%u_t+\sum_{i=1}^nA^i(x,u)u_i=0
%\end{equation}
and, instead of referring to (\ref{Eq:InPDE0}), we will use 
\begin{equation}\label{Eq:InPDE02}
	\sum_{i=1}^{p}\sum_{\beta=1}^q A^{i\alpha}_\beta(x,{u}){u}^\beta_i=0,\qquad \alpha=1,\ldots,q,
\end{equation}
where we assume  a generic number $p$ of independent variables, which will be useful in what follows.

Consider an underdetermined quasilinear system (\ref{Eq:InPDE}). We can slightly generalise and rewrite such a system as follows
$$
\sum_{\mu=1}^p\sum_{i=1}^{q_h}A_i^{s\mu}(x,u)\lambda_{\mu}\gamma_1^i+\sum_{\mu=1}^p\sum_{i=q_h+1}^qA^{s\mu}_i(x,u)\lambda_{\mu}\gamma^i_2=b^s(x,u),\quad s=1,\ldots,m,
$$
where $1\leq q_h\leq q$ is a certain value used to divide the quasilinear system (\ref{Eq:InPDE}) into two parts. 
The vector $\gamma=(\gamma_1,\gamma_2)$ has $q$ components, while
$$
\gamma_1=(\gamma^{q_1})_{q_1=1,\ldots,q_h},\qquad \gamma_2=(\gamma^{q_2})_{q_2=q_h+1,\ldots,q}
$$
and we define
$$
(A_1\lambda ,A_2\lambda )=\left(\left(\sum_{\mu=1}^pA^{s\mu}_{q_1}\lambda_\mu\right)^{s_1},\left(\sum_{\mu=1}^pA^{s\mu}_{q_2}\lambda_\mu\right)^{s_2}\right),
$$
where $q_1=1,\ldots,q_h,q_2=q_h+1,\ldots,q$, to be an $m\times q$ matrix which is a proper decomposition of the matrix $\sum_{\mu=1}^pA^{s\mu}_j\lambda_{\mu}$ for $s=1,\ldots,m$ and $j=1,\ldots,q$ into two submatrices of dimensions $m\times q_1$ and $m\times q_2$.

If the  matrix $A_1\lambda$ is a nonsingular $m\times m$ matrix, then one can determine a quantity
$$
\gamma_{1}=(A_1\lambda )^{-1}(b-(A_2\lambda )\gamma_2)
$$ 
and there exists a correspondence $\gamma\otimes \lambda\mapsto (\gamma_2,\lambda)$, which determines a map on the set of simple integral elements for which the domain is formed by these elements that have the nonsingular matrix $A_1\lambda$.
All possible divisions of such a type determine an atlas composed from $(l m)$ maps covering the set of simple integral elements.

Note that it is remarkable that one can divide the set of simple integral elements $\gamma_{(s)}\otimes \lambda^{(s)}$ via strata numbered by rank of matrices $\sum_{\mu=1}^pA^{s\mu}_j\lambda_{\mu}$. Regular elements form strata of the highest dimension. One can determine atlases for other strata analogously.
\section{Reduction to homogeneous first-order systems}\label{Se:ApplIn}
As an application of the approach of the previous section, let us study the partial differential equations
% \begin{equation}\label{Eq:Ex1}
	% u_t+(1+\beta^2x^2)u_{xx}=0,\qquad u\in \mathbb{R}.
	% \end{equation}
% To study this differential equation, we introduce
% $$
% A^t=\left[\begin{array}{cc}0&0\\0&-1\end{array}\right],\quad A^x=\left[\begin{array}{cc}0&1+\beta^2x^2\\1&0\end{array}\right],\quad b=\left[\begin{array}{c}-a\\0\end{array}\right],\quad V=\left[\begin{array}{c}a\\c\end{array}\right],
% $$
% where $a=u_t$ and $c=u_x$. Then, (\ref{Eq:Ex1}) can be written as
% $$
% A^tV_t+A^xV_x=b.
% $$
% By defining 
% $$
% M=\frac{1}{-a}\left[\begin{array}{cc}1&0\\0&-a\end{array}\right],
% $$
% one obtains that 
% $$
% MA^tV_t+MA^xV_x=Mb
% $$
% can be rewritten in the new variables
% $$
% \widetilde{V}=V-\left[\begin{array}{c}y\\0
	% \end{array}\right],\quad \widetilde{V}=\left[\begin{array}{c}\widetilde{a}\\\widetilde{c}\end{array}\right]
% $$
% as
% \begin{equation}\label{Eq:P}
	% \left[\begin{array}{cc}0&0\\0&-1\end{array}\right]\widetilde{V}_t+\left[\begin{array}{cc}0&\frac{1+\beta^2x^2}{\widetilde{a}+y}\\1&0\end{array}\right]\widetilde{V}_x+\left[\begin{array}{cc}1&0\\0&1\end{array}\right]\widetilde{V}_{y}=\left[\begin{array}{c}0\\0\end{array}\right].
	% \end{equation}
%
%Let us now consider a slightly more general example. Let us study the equation
\begin{equation}\label{Ex:Examples}
	u_t-(1+\beta^2x^2)^\kappa u_{xx}=0,\qquad u\in \mathbb{R},\qquad \beta\in \mathbb{R},\qquad \kappa\in \mathbb{N}\cup\{0\},
\end{equation}
which reproduce the heat equation and the equation of Brownian motion as particular cases. 
We introduce
$$
A^t=\left[\begin{array}{cc}0&0\\0&-1\end{array}\right],\quad A^x=\left[\begin{array}{cc}0&(1+\beta^2x^2)^\kappa \\1&0\end{array}\right],\quad b=\left[\begin{array}{c}a\\0\end{array}\right],\quad V=\left[\begin{array}{c}a\\c\end{array}\right],
$$
where $a=u_t$ and $c=u_x$. Then, equation (\ref{Ex:Examples}) can be written as
$$
A^tV_t+A^xV_x=b.
$$
According to (\ref{Eq:M}), in this case, the matrix $M$ takes the form
$$
M=\frac{1}{a}\left[\begin{array}{cc}1&0\\0&a\end{array}\right].
$$
One obtains that  the inhomogeneous system of partial differential equations
$$
MA^tV_t+MA^xV_x=Mb
$$
can be rewritten in the new variables
$$
\widetilde{V}=V-\left[\begin{array}{c}y\\0
\end{array}\right]
$$
as a homogeneous system
\begin{equation}\label{Eq:P}
	\left[\begin{array}{cc}0&0\\0&-1\end{array}\right]\widetilde{V}_t+\left[\begin{array}{cc}0&\frac{(1+\beta^2x^2)^\kappa }{\widetilde{a}+y}\\1&0\end{array}\right]\widetilde{V}_x+\left[\begin{array}{cc}1&0\\0&1\end{array}\right]\widetilde{V}_{y}=\left[\begin{array}{c}0\\0\end{array}\right].
\end{equation}

Let us consider another simple example:  the wave equation
\begin{equation}\label{Eq:Trautman}
	(\partial^2_t-\partial^2_x+k^2)u=0,\qquad k\in \mathbb{R},
\end{equation}
appearing in relativity problems concerning the existence of gravitational waves. To study it as a linear homogeneous partial differential equation, we introduce new coordinates $v_0=u_t$ and $v_1=u_x$. Hence, (\ref{Eq:Trautman}) is equivalent to 
$$
\partial_t v_0-\partial_x v_1=-k^2u,\qquad \partial_tu=v_0,\qquad \partial_xu=v_1.
$$
Consequently, one can write
$$
A^t=\left[\begin{array}{ccc}1&0&0\\0&0&1\\0&0&0\\\end{array}\right],\,\, A^x=\left[\begin{array}{ccc}0&-1&0\\0&0&0\\0&0&1\\\end{array}\right],\,\, b=\left[\begin{array}{c}-k^2u\\v_0\\v_1\end{array}\right],\,\, V=\left[\begin{array}{c}v_0\\v_1\\u\end{array}\right]
$$
and therefore 
\begin{equation}\label{Eq:Eq2}
	A^tV_t+A^xV_x=b.
\end{equation}
We consider the matrix 
$$
M=\frac{1}{-k^2u}\left[\begin{array}{ccc}1&0&0\\-v_0&-k^2u&0\\-v_1&0&-k^2u\end{array}\right]
$$
and the change of variables
$$
\widetilde{V}=V-\left[\begin{array}{c}\widehat{x}\\0\\0
\end{array}\right],
$$
which maps (\ref{Eq:Eq2})  onto
\begin{multline*}
	\frac{1}{-k^2u}\left[\begin{array}{ccc}1&0&0\\-\widehat{v_0}-\widehat{x}&0&-k^2u\\-v_1&0&0\end{array}\right]\widetilde{V}_t-\frac{1}{k^2u}\left[\begin{array}{ccc}0&-1&0\\0&\widehat{v_0}+\widehat{x}&0\\0&v_1&-k^2u\end{array}\right]\widetilde{V}_x\\+\left[\begin{array}{ccc}1&0&0\\0&1&0\\0&0&1\end{array}\right]\widetilde{V}_{\widehat{x}}=\left[\begin{array}{c}0\\0\\0\end{array}\right].
\end{multline*}

\section{On $k$-wave solutions}\label{Se:lw}
A particular solution $u:\mathbb{R}^p\rightarrow \mathbb{R}^q$ of the quasi-linear homogeneous system (\ref{Eq:InPDE02}) such that $T_xu:T_x\mathbb{R}^p\rightarrow T_{u(x)}\mathbb{R}^q$ has rank $k$ at every point of $\mathbb{R}^p$ is called a $k$-{\it wave solution}. To simplify the notation, we will use the Einstein summation convention. In this notation, indices will be summed between the proper limits for each type of variable, e.g. the independent variables are summed over their indices ranging from $1$ to $p$, while dependent variables are summed over their indices from $1$ to $q$. 

Let us consider solutions of the hydrodynamic equations given by
$$
u(x)=f\left(\lambda^{(1)}_j(u(x))x^j,\ldots,\lambda^{(k)}_j(u(x))x^j\right)
$$
for certain functions $f:\mathbb{R}^l\rightarrow \mathbb{R}^q$, $\lambda^{(i)}:\mathbb{R}^q\rightarrow \mathbb{R}^p$, with $i=1,\ldots,k$ such that  $\lambda^{(1)}\wedge\ldots\wedge \lambda^{(k)}\neq0$ on $\mathbb{R}^q$. 
The family $\lambda^{(1)},\ldots,\lambda^{(k)}$ is called a $k$-{\it wave covector} and  the functions $r^{(s)}:\mathbb{R}^p\times\mathbb{R}^q\rightarrow \mathbb{R}$ of the form
$$
r^{(s)}(x,u):=\lambda^{(s)}_j(u)x^j,\qquad s=1,\ldots,k,
$$
are referred to as {\it $k$-order Riemann invariants}. 
Thus, the condition
$$u(x)
=f(r(x,u(x))),\quad r(x,u(x))=( r^{(1)}(x,u(x)),\ldots,r^{(k)}(x,u(x)))$$
gives rise to a matrix of partial derivatives with entries
\begin{equation}\label{Eq:MatrixPartialDeri}
	\frac{\partial u^\alpha}{\partial x^i}(x)=\frac{\partial f^\alpha}{\partial r^{(s)}}(r(x,u))\left[\lambda_i^{(s)}(u(x))+x^j\frac{\partial \lambda_j^{(s)}}{\partial u^\beta}(u(x))\frac{\partial u^\beta}{\partial x^i}(x)\right]
\end{equation}
and then
{\small $$
	\left[\delta_\alpha^\beta-x^j\frac{\partial f^\alpha}{\partial r^{(s)}}(r(x,u(x)))\frac{\partial \lambda_j^{(s)}}{\partial u^\beta}(u(x))\right]\frac{\partial u^\beta}{\partial x^i}(x)=\lambda_i^{(s)}(u(x))\frac{\partial f^\alpha}{\partial r^{(s)}}(r(x,u(x))).
	$$}
If $\Psi(x,u)$ is an invertible $q\times q$ matrix whose entries are given by the functions
$$
\Psi^\alpha_\beta(x,u):=\delta_\alpha^\beta-x^j\frac{\partial f^\alpha}{\partial r^{(s)}}(r(x,u))\frac{\partial \lambda_j^{(s)}}{\partial u^\beta}(u),
$$
we get that $\Psi(0,u)$ is the $q\times q$ identity matrix and $\Psi(x,u)$ is invertible if $x$ is close enough to zero. On the other hand,
\begin{multline*}
	\Psi^\alpha_\beta(x,u)\frac{\partial f^\beta}{\partial r^{(s)}}(r(x,u))\\
	=\frac{\partial f^\alpha}{\partial r^{(s)}}(r(x,u))-x^j\frac{\partial f^\alpha}{\partial r^{(v)}}(r(x,u))\frac{\partial \lambda_j^{(v)}}{\partial u^\beta}(u)\frac{\partial f^\beta}{\partial r^{(s)}}(r(x,u))\\=\frac{\partial f^\alpha}{\partial r^{(s)}}(r(x,u))-\frac{\partial f^\alpha}{\partial r^{(v)}}(r(x,u))\frac{\partial r^{(v)}}{\partial u^\beta}(r(x,u))\frac{\partial f^\beta}{\partial r^{(s)}}(r(x,u)).
\end{multline*}
Therefore,
\begin{multline*}
	\Psi^\alpha_\beta(x,u)\frac{\partial f^\beta}{\partial r^{(s)}}(r(x,u))\\=\left[\delta_s^v-\frac{\partial f^\beta}{\partial r^{(s)}}(r(x,u))\frac{\partial r^{(v)}}{\partial u^\beta}(r(x,u))\right]\frac{\partial f^\alpha}{\partial r^{(v)}}(r(x,u))=[\Phi(x,u)]^v_s\frac{\partial f^\alpha}{\partial r^{(v)}}(r(x,u)),
\end{multline*}
where
$$
[\Phi(x,u)]^v_s:=\delta_s^v-\frac{\partial f^\beta}{\partial r^{(s)}}(r(x,u))\frac{\partial r^{(v)}}{\partial u^\beta}(r(x,u)).
$$
Therefore,
$$
\frac{\partial f^\alpha}{\partial r^{(v)}}(r(x,u))=[\Phi^{-1}(x,u)]^s_v\Psi^\alpha_\beta(x,u)\frac{\partial f^\beta}{\partial r^{(s)}}(r(x,u)).
$$
From this, one gets that
\begin{multline*}
	\frac{\partial u^\alpha}{\partial x^i}(x)=(\Psi^{-1})^\alpha_\beta(x,u)\lambda^{(s)}_i(u(x))\frac{\partial f^\beta}{\partial r^{(s)}}(r(x,u))\\=(\Psi^{-1})^\alpha_\beta(r(x,u))\lambda_i^{(s)}[\Phi^{-1}(x,u)]^s_t\Psi^\beta_\delta(x,u)\frac{\partial f^\delta}{\partial r^{(t)}}(r(x,u))\\=[\Phi^{-1}(x,u)]^t_s\lambda_i^{(s)}(u(x))\frac{\partial f^\alpha}{\partial r^{(t)}}(r(x,u)).
\end{multline*}
Consider a function of the form $\zeta:u\in \mathbb{R}^q\mapsto (\zeta^1(u),\ldots,\zeta^p(u))\in \mathbb{R}^p$ so that $\langle \lambda(u),\zeta(u)\rangle=0$. Then,
$$
\zeta^j(u(x))\frac{\partial u^\alpha}{\partial x^j}(x)=[\Phi^{-1}(x,u)]^t_s\zeta^j(u(x))\lambda_j^{(r)}(u(x))\frac{\partial f^\alpha}{\partial r^{(t)}}(r(x,u))=0
$$
and the coordinates $u^1(x),\ldots,u^q(x)$ of a solution $u(x)$ are invariant relative to the vector fields 
$$
X_{\zeta,u(x)}=\zeta^i(u(x))\frac{\partial}{\partial x^i}.
$$
In fact,
\begin{multline*}
	(X_{\zeta,u(x)} u^\alpha)(x)=\zeta^i(u(x))\frac{\partial u^\alpha}{\partial x^i}(x)\\=\zeta^i(u(x))[\Phi^{-1}(x,u(x))]^t_s\zeta^j(u(x))\lambda_j^{(r)}(u(x))\frac{\partial f^\alpha}{\partial r^{(t)}}(r(x,u))=0. 
\end{multline*}
Hence, the coordinates $u^\alpha(x)$ of any particular solution $u(x)$  of the hydrodynamic equations (\ref{Eq:InPDE02}) are constants of the motion of $X_{\zeta,u(x)}$ and $u(x)$ is constant along its integral characteristic curves.

\section{Conditions for the existence of  $k$-wave solutions}\label{Se:genkwave}
We hereafter assume that $b=0$ in (\ref{Eq:InPDE}). This section generalises the theory developed in \cite{GL91,GSW00,GV91} to the case of {\it $k$-waves} for the hydrodynamic equations (\ref{Eq:InPDE}) whose coefficients  $A^1,\ldots,A^p$ depend on dependent and independent variables.  This dependence was missing in the theory of $k$-wave solutions  for hydrodynamic equations devised in \cite{GV91}. Meanwhile, in \cite{GL91} the same hydrodynamic equations as in our present work were studied, but only two-waves,  were dealt with. 

Most results given in the previous sections can be extended to the most general case of $k$-wave solutions to be analysed in this section. Proofs given hereafter essentially follow the  ideas described in the previous sections, but they are technically much more complicated. Moreover, to keep the calculations easy to read, we skip certain details like the dependence of each type of function. Nevertheless, these details can be inferred by analysing the particular cases given in the previous sections.

Recall that a  {$k$-wave solution} of the hydrodynamic equations (\ref{Eq:InPDE02}) is a solution  $u:\mathbb{R}^p\rightarrow \mathbb{R}^q$ such that the tangent mapping $T_xu:T_x\mathbb{R}^p\rightarrow T_{u(x)}\mathbb{R}^q$ has rank $k$ at every $x\in \mathbb{R}^p$. Then, the tangent map can be written as
\begin{equation}\label{Eq:lsolution}
	T_xu=\sum_{\sigma=1}^k\xi^{\sigma}(x)\gamma_{(\sigma)}(x,u(x))\otimes \lambda^{(\sigma)}(x,u(x)),\quad x\in \mathbb{R}^p,\quad u(x)\in\mathbb{R}^q,
\end{equation}
where we assume $\xi:x\in \mathbb{R}^p\mapsto (\xi^1(x),\ldots,\xi^k(x))\in \mathbb{R}^{k}$, $\lambda^{(\sigma)}:(x,u)\in \mathbb{R}^{p+q}\mapsto (\lambda^{(\sigma)}_1(x,u),\ldots,\lambda^{(\sigma)}_p(x,u))\in \mathbb{R}^{p*}$,  and $\gamma_{(\sigma)}:\mathbb{R}^{p+q}\rightarrow \mathbb{R}^q$ for $\sigma=1,\ldots,k$. Since $T_xu$ has rank $k$ by assumption, $\gamma_{(1)}\wedge\ldots\wedge \gamma_{(k)}\neq 0$ and $\lambda^{(1)}\wedge\ldots\wedge \lambda^{(k)}\neq 0$ on $\mathbb{R}^{p+q}$. It is worth stressing that $T_xu$ is denoted by $du$ in other works \cite{GL91,GV91}. Recall also that the functions $u^1,\ldots,u^q,\xi^1,\ldots,\xi^k$ are considered as unknown functions of $x^1,\ldots,x^p$. A $k$-wave solution of (\ref{Eq:lsolution}), understood as a system of partial differential equations relative to $u(x)$, satisfies the hydrodynamic system (\ref{Eq:InPDE02}).

From now on, the dependence of $\xi$, $\lambda^{(1)},\ldots,\lambda^{(k)}$ and $\gamma_{(1)},\ldots,\gamma_{(k)}$, on their domain variables will be omitted to simplify the notation.   We also keep on using the Einstein summation convention as in the previous sections.  Moreover, Latin indices stand for the coordinates in $\mathbb{R}^p$. The Greek index $\sigma$ is used to sum over the $k$ components of a $k$-wave, while the remaining Greek indices are employed to sum over the coordinates in $\mathbb{R}^q$. Finally, we stress that we are interested in decompositions (\ref{Eq:lsolution}) such that, for fixed values of the simple elements $\gamma_{(1)}(x,u)\otimes \lambda^{(1)}(x,u), \ldots, \gamma_{(k)}(x,u)\otimes \lambda^{(k)}(x,u)$, the values of $\xi^1(x_0),\ldots,\xi^k(x_0)$ can be arbitrarily changed at a point $x_0\in \mathbb{R}^p$ without changing the fact that the associated function $u(x)$ is a solution of (\ref{Eq:InPDE02}).

Recall that the form of the vector functions $\gamma_{(1)},\ldots,\gamma_{(k)}$ depends on the covector functions $\lambda^{(1)},\ldots,\lambda^{(k)}$ so that $u(x)$ is a solution of the system (\ref{Eq:InPDE02}). In particular, the wave relation
\begin{equation}
	%	\label{Rel:GammaDelta}
	A^{i\alpha}_\beta\xi^\sigma\gamma^\beta_{(\sigma)}\lambda^{(\sigma)}_i=0,\qquad \alpha=1,\ldots,q, 
\end{equation}
holds. Since the above must be satisfied for every value of $\xi^1,\ldots,\xi^k$ at a certain fixed point in $\mathbb{R}^p$ and fixed $\lambda^{(\sigma)},\gamma_{(\sigma)}$, it follows that
\begin{equation}
	\label{Rel:GammaDelta}
	A^{i\alpha}_\beta\gamma^\beta_{(\sigma)}\lambda^{(\sigma)}_i=0,\qquad \alpha=1,\ldots,q, \quad \sigma=1,\ldots,k.
\end{equation}
In other words, if we fix the values of $\lambda^{(\sigma)}$, then 
\begin{equation}%\label{Eq:lambda}
	\gamma_{(\sigma)}\in \bigcap_{\alpha=1}^q\ker (\lambda^{(\sigma)}_iA_\beta^{i\alpha})\neq0,
	\qquad \sigma=1,\ldots,k.\end{equation}
Our main aim in this section is to extend  to  $k$-waves the result that the single waves of (\ref{Eq:InPDE02}) can be written as implicit rank-one solutions as given in the following proposition. Its proof, with less detail than the one below, can be found in \cite{GL91}.

\begin{prop} {\bf (Simple Riemann wave solutions)} Suppose that there exists a map $\varphi:\mathbb{R}^{p+q}\rightarrow \mathbb{R}$ such that the set of implicitly defined relations between the variables $u,x$, and $s$, namely
	\begin{equation}\label{Eq:eqeq}
		u=f(s),\qquad s=\varphi(x,u),
	\end{equation}
	can be solved so that $s$ and $u$ can be given as a graph over an open subset $D\subset E$, i.e. one has $s=s(x)$ and $u=f(x)$ for $x$ in an open subset $A\subset \mathbb{R}^p$. Assume that $\gamma(x,u)=\gamma(x',u)$ whenever    $s(x)=s(x')$. 
	Suppose that $\Gamma:u=f(s)$ is a characteristic curve obtained by solving the system of ODEs 
	$$
	\frac{du}{ds}=\alpha(s)\widetilde{\gamma}(s,u),
	$$
	where $\alpha$ is an arbitrary function of $s$ and $\widetilde{\gamma}(s(x),u)=\gamma(x,u)$ for $u\in \mathbb{R}^q$ and $x\in A$.  Then, (\ref{Eq:eqeq}) constitutes an exact (Riemann-type wave) solution of the quasilinear system (\ref{Eq:InPDE02}) with coefficients depending on $x$ and $u$.
\end{prop}
\begin{proof}
	The proof is obtained directly by the differentiation of (\ref{Eq:eqeq}), i.e. there exists $s(x)$, univocally determined, such that 
	$$
	s(x)=\varphi(x,f(s(x)))\Rightarrow  (ds)(x)=\frac{(d_x\varphi)(x,u(x))}{1-\sum_{\alpha=1}^q\frac{\partial\varphi}{\partial u^\alpha}(x,u(x))\frac{df^\alpha}{ds}(s(x))}.
	$$
	This shows that
	\begin{equation}\label{Eq:Cat}
		du(x)=\frac{\alpha(s(x))\widetilde{\gamma}(s(x),u(x)))\otimes (d_x\varphi)(x,u(x))}{1-\sum_{\alpha=1}^q\frac{\partial\varphi}{\partial u^\alpha}(x,u(x))\frac{df^\alpha}{ds}(s(x))},
	\end{equation}
	where $\widetilde{\gamma}(s(x),u(x))=\gamma(x,u(x))$.
	Hence,
	\begin{equation}\label{Eq:SE}
		du(x)=f(x)\gamma(x,u(x))\otimes \lambda(x,u(x)),
	\end{equation}
	for a certain function $f(x)$ and $du$
	becomes a simple integral element (\ref{Eq:SE}) of the quasilinear system (\ref{Eq:bezx}), where we exclude the gradient catastrophe where 
	$$
	\sum_{\alpha=1}^q\frac{\partial\varphi}{\partial u^\alpha}(x,u(x))\frac{df^\alpha}{ds}(s(x))=1.
	$$
\end{proof}
%The other way around, if $u(x)$ is constant along the integral curves of a vector field $\xi^i\partial/\partial x_i$ such that $\langle \xi,\lambda\rangle=0$, then 
%$$
%\xi_{(s)}^i(u)\frac{\partial}{\partial x_i},\qquad  s=1,\ldots,p-1,
%$$
%and $\lambda_i\partial_i$ must be such that $u^\alpha(x)=\lambda^\alpha_ix^i$.

%Let us consider now a vector.
It is worth noting that the previous proposition implies that $(d_x\gamma)(x,u)$ is proportional to $ds(x)$. Moreover, the solution makes sense when the denominator in the value of $ds(x)$ is different from zero, which is just the condition of solvability of the system (\ref{Eq:eqeq}). Note also that $\gamma$ must take a constant value, for each fixed $u\in \mathbb{R}^q$, on the submanifold where $s(x)$ takes a constant value. This is a difference relative to the formalism of simple waves where $\lambda$ depends only on the dependent variables.

%of the system
%\begin{equation}\label{Eq:SysNew}
% \begin{cases}u(x)=f(\tau(x)), \\\tau(x)=\phi(x,f(\tau(x))),\end{cases} 
%\end{equation}
%for  certain functions $f:\mathbb{R}^k\rightarrow \mathbb{R}^q$ and $\phi:\mathbb{R}^{p+q}\rightarrow \mathbb{R}^k$ such that $\gamma(x,u)$ is a constant along on the submanifolds where $\tau(x)$ takes a constant value.% Such solutions are called {\it rank one} or {\it simple} wave solutions.

Note that from the the relation (\ref{Eq:Cat}), it follows that on the hypersurface $\mathcal{M}$ given by the   relations 
$$
u=f(s),\qquad s=\varphi(x,u),\qquad \sum_{\alpha=1}^q\frac{\partial \varphi}{\partial u^\alpha}(x,u(x))\frac{df^\alpha}{ds}(s(x))=1,
$$
the gradient of the function $s(x)$ becomes infinite. The solution can lose its sense (becomes infinite) on the hypersurface $\mathcal{M}$. Consequently, some types of discontinuities, namely {\it shock waves}, can appear.

For $k$-waves, the tangent  map $T_xu:T_x\mathbb{R}^p\rightarrow T_{u(x)}\mathbb{R}^q$ can be understood as an element of $(\mathbb{R}^{p})^*\otimes \mathbb{R}^q$. Moreover, each solution, $u(x)$, can be understood as a zero-form on $\mathbb{R}^p$ with values in $\mathbb{R}^q$. As such, the differential of $u$, assuming that it has rank $k$, takes the form  (\ref{Eq:lsolution}). 

To simplify the notation in what follows, we shall employ two differentials. Given $f\in C^\infty(\mathbb{R}^{p+q},\mathbb{R}^q)$, namely a function on $\mathbb{R}^{p+q}$ taking values in $\mathbb{R}^q$, we write
$$
d_xf= \frac{\partial f}{\partial x^i}dx^i,\qquad df=d_xf+ \frac{\partial f}{\partial u^\alpha}d_xu^\alpha(x).
$$
The second differential, $d^2u(x)$, must vanish and
\begin{equation}
	\label{eq:2}  d_x^2u(x)= \xi^{\sigma} d\gamma_{(\sigma)} \wedge \lambda^{(\sigma)} +\gamma_{(\sigma)} \otimes (d\xi^{\sigma} \wedge \lambda^{(\sigma)}+\xi^\sigma d\lambda^{(\sigma)} )=0.
\end{equation}
Moreover, we define
$$
X^i_{u^\alpha}=\frac{\partial X^i}{\partial u^\alpha},\quad X^i_Y=\frac{\partial X^i}{\partial u^\alpha}Y^\alpha,\quad \forall X^i\in C^\infty(\mathbb{R}^{p+q}),\quad \forall Y\in \mathbb{R}^q.
$$
%Moreover, we also define the differential on the variable $x\in\mathbb{R}^p$ by
%$$
%d_x\gamma_{(m)}:=\frac{\partial\gamma_{(m)}}{\partial x^i}dx^i.
%$$
By using (\ref{Eq:lsolution}), we obtain
$$
\begin{gathered}
	d\gamma_{(m)}=d_x\gamma_{(m)}+\gamma_{(m),\alpha}d_xu^\alpha=d_x\gamma_{(m)}+\xi^\sigma\gamma_{(m),\gamma_{(\sigma)}}\otimes \lambda^{(\sigma)},\\
	d\lambda^{(m)}=du^\alpha\wedge \lambda^{(m)}_{,u^\alpha}+d_x\lambda^{(m)}=d_x\lambda^{(m)}+\xi^\sigma\lambda^{(\sigma)}\wedge \lambda^{(m)}_{,\gamma_{(\sigma)}},
\end{gathered}
$$
for $m=1,\ldots,k$. Substituting the above expressions in (\ref{eq:2}), we get
\begin{multline}\label{Eq:Exp}
	\frac 12\xi^{\sigma}\xi^{\sigma'}(\gamma_{(\sigma)},\gamma_{(\sigma')})_u\otimes \lambda^{(\sigma')}\wedge\lambda^{(\sigma)}+\xi^\sigma\xi^{\sigma'}\gamma_{(\sigma)}\otimes \lambda^{(\sigma')}\wedge \lambda_{,\gamma_{(\sigma')}}^{(\sigma)}
	\\+\xi^\sigma d_x\gamma_{(\sigma )}\wedge\lambda^{(\sigma )}+\gamma_{(\sigma)}\otimes d\xi^\sigma\wedge\lambda^{(\sigma)}+\gamma_{(\sigma)}\otimes \xi^\sigma d_x\lambda^{(\sigma)},
\end{multline}
which has to be satisfied whenever (\ref{Eq:lsolution}) holds, 
where
%$$
%[\gamma_{(m)},\gamma_{(q)}]_{(u,\lambda)}=(\gamma_{(m)},\gamma_{(q)})_u+\gamma_{(q),\lambda^{(q)}}\lambda^{(q)}_{,\gamma_{(m)}}-\gamma_{(m),\lambda^{(m)}}\lambda^{(m)}_{,\gamma_{q}}
%$$
%and 
$(\gamma_{(\sigma)},\gamma_{(\sigma')})_u$, which can be understood as a commutator, is an element of $C^\infty(\mathbb{R}^{p+q},\mathbb{R}^q)$ with coordinates
$$
(\gamma_{(\sigma)},\gamma_{(\sigma')})^\beta_u=\gamma_{(\sigma')}^\alpha\frac{\partial \gamma_{(\sigma)}^\beta}{\partial u^\alpha}-\gamma_{(\sigma)}^\alpha\frac{\partial \gamma_{(\sigma')}^\beta}{\partial u^\alpha}=\gamma^{\beta}_{(\sigma),\gamma_{(\sigma')}}-\gamma^{\beta}_{(\sigma'),\gamma_{(\sigma)}},\quad \forall \beta=1,\ldots,q.
$$
For each fixed value of $x\in \mathbb{R}^p$, the functions $\gamma_{(1)},\ldots,\gamma_{(k)}$ can be understood as the coordinates of  the $x$-dependent vector fields  on (see \cite{Dissertationes} for further details on this notion) $\mathbb{R}^{q}$  of the form
\begin{equation}
	\label{Eq:sigmaVector}
	\chi_\sigma:=\gamma_{(\sigma)}^\beta(x,u)\frac{\partial}{\partial u^\beta},\qquad \sigma=1,\ldots,k.
\end{equation}
Hence, $(\gamma_{(\sigma)},\gamma_{(\sigma')})_u$ can be understood, for each fixed value $x\in \mathbb{R}^p$, as the Lie bracket of the vector fields on $\mathbb{R}^q$ related to $\gamma_{(\sigma)}$ and $\gamma_{(\sigma')}$.

Let us consider an $x$-parametrised differential one-form on $\mathbb{R}^q$ satisfying $\omega(x,u)\in {\rm span}\{ \gamma_1(x,u),\ldots,\gamma_r(x,u)\}^\circ$, i.e. a mapping $\omega:\mathbb{R}^{p+q}\rightarrow T^*\mathbb{R}^q$ such that $\omega(x,\cdot):\mathbb{R}^q\rightarrow T^*\mathbb{R}^q$ is a standard differential one-form on $\mathbb{R}^q$ for each $x\in \mathbb{R}^p$ and $\omega(x,\cdot)$ vanishes on $\gamma_{(1)}(x),\ldots,\gamma_{(r)}(x)$.  Composing (\ref{Eq:Exp}) with $\omega$, one obtains
\begin{equation}\label{Eq:In}
	\frac 12\xi^{\sigma}\xi^{\sigma'}\langle \omega,(\gamma_{(\sigma)},\gamma_{(\sigma')})_u\rangle \lambda^{(\sigma')}\wedge \lambda^{(\sigma)}+\xi^\sigma\langle \omega,d_x\gamma_{(\sigma)}\rangle\wedge \lambda^{(\sigma)}=0.
\end{equation}
By our initial assumption, the above must hold for all possible values of the functions $\xi^1,\ldots,\xi^k$ at any fixed $x\in \mathbb{R}^p$. Then,
\begin{equation}\label{Eq:InvDis}
	(\gamma_{(\sigma)},\gamma_{(\sigma')})_u=\sum_{\sigma''=1}^k\tau^{\sigma''}_{\sigma \sigma'}\gamma_{(\sigma'')},\qquad \forall \sigma,\sigma'=1,\ldots,k,
\end{equation}
for certain functions $\tau^{\sigma''}_{\sigma \sigma'}\in C^\infty(\mathbb{R}^{p+q})$, where $\sigma,\sigma',\sigma''=1,\ldots,k$. Since   $\gamma_{(1)},\ldots,\gamma_{(k)}$ can be understood as $x$-dependent vector fields on $\mathbb{R}^q$ of the form (\ref{Eq:sigmaVector})
and in view of (\ref{Eq:InvDis}), it turns out that, for every fixed $x\in \mathbb{R}^p$, the vector fields $\chi_1(x),\ldots,\chi_k(x)$ span an involutive distribution on $\mathbb{R}^q$. Since $\{\gamma_{(1)},\ldots,\gamma_{(k)}\}$ are linearly independent by construction at every $(x,u)\in \mathbb{R}^{p+q}$, the $x$-dependent vector fields on $\mathbb{R}^{q}$ given by $\{\chi_1,\ldots,\chi_k\}$ span a distribution $\mathcal{D}_x(u)={\rm span}\{ \chi_1(x,u),\ldots,\chi_k(x,u)\}$, for $u\in \mathbb{R}^q$, of rank $k$ on $\mathbb{R}^q$ for every $x\in \mathbb{R}^p$. Hence, $\mathcal{D}_x$ can be integrated on $\mathbb{R}^q$ for every $x\in \mathbb{R}^p$ and there exists a family of $x$-parametric foliations of $\mathbb{R}^q$ of $k$-dimensional leaves. %Moreover, due to the 
%Fr\"obenius Theorem, $\mathcal{D}_x$ can be generated by $k$ commuting and linearly independent vector fields on $\mathbb{R}^q$ obtained as linear combinations with coefficients depending on independent and dependent variables of $\gamma_{(1)},\ldots,\gamma_{(k)}$. Moreover, note that the new $x$-vector fields on $\mathbb{R}^q$ will still satisfy equations (\ref\xi^a). Consequently, one can assume, without loss of generality, that $\gamma_{(1)},\ldots,\gamma_{(k)}$, thought of as vector fields on $\mathbb{R}^q$ for every $x\in\mathbb{R}^p$, commute between themselves and the $x$-parametric functions $\tau_{\sigma\sigma'}^{\sigma''}$ vanish for $\sigma,\sigma',\sigma''=1,\ldots,k$ and every $x$. 

Expression (\ref{Eq:lsolution}) demonstrates that the tangent space to each solution $u(x)$ of the hydrodynamic equations at $x\in \mathbb{R}^p$  belongs to  $\mathcal{D}_x$. It is worth stressing  that the distribution $\mathcal{D}_x$ may be different at different values of $x$. %Nevertheless, we cannot still ensure that every solution $u(x)$ takes values in a leaf of the distribution $\mathcal{D}_x$.

Recalling  (\ref{Eq:In}) once again, one obtains
\begin{equation}\label{Eq:ConditionProblem}
	\langle w,d_x\gamma_{(m)}\rangle\wedge \lambda^{(m)}=0,\qquad m=1,\ldots,k.
\end{equation}
Let us assume that
\begin{equation}
	\label{eq:gamma1}
	d_x\gamma_{(m)}=\beta^\sigma_m\gamma_{(\sigma)}\otimes\lambda^{(\sigma)}+\pi_{(m)}\otimes \lambda^{(m)},\qquad m=1,\ldots,k,
\end{equation}
where $\pi_{(m)}$ is any $x$-parametrised vector field on $\mathbb{R}^q$, we sum over the $\sigma$ index and we do not sum over the $m$, and the $\beta^\sigma_m$ are $x$-parametrised functions on $\mathbb{R}^q$. Note that condition (\ref{eq:gamma1}) is not necessary for (\ref{Eq:ConditionProblem}) to hold, since it can be proved that a term $\gamma_{(k)}(x,u)\otimes \vartheta(x,u)$, where $\vartheta(x,u)$ is any $u$-parametrised differential one-form on $\mathbb{R}^p$, is a solution of (\ref{eq:gamma1}).

By substituting (\ref{eq:gamma1}) in (\ref{Eq:Exp}), one obtains
\begin{multline}\label{eq:P2}
	0= \frac 12\xi^{\sigma}\xi^{\sigma'}(\gamma_{(\sigma)},\gamma_{(\sigma')})_u\otimes \lambda^{(\sigma')}\wedge\lambda^{(\sigma)}+\xi^\sigma\xi^{\sigma'} \gamma_{(\sigma')}\otimes \lambda^{(\sigma)}\wedge \lambda_{,\gamma_{(\sigma)}}^{(\sigma')} 
	\\+\xi^\sigma\beta^{\sigma'}_\sigma\gamma_{(\sigma')}\otimes\lambda^{(\sigma')}\wedge\lambda^{(\sigma)}+\gamma_{(\sigma)}\otimes d\xi^\sigma\wedge\lambda^{(\sigma)}+\gamma_{(\sigma)}\otimes \xi^\sigma d_x\lambda^{(\sigma)}.
\end{multline}
Let us now consider a set of $x$-dependent differential forms on $\mathbb{R}^q$ such that
$$
\langle \omega_{\sigma'},\gamma_{(\sigma')}\rangle=\delta_{\sigma'}^\sigma,\qquad \sigma',\sigma=1,\ldots,k.
$$ Applying $\omega_s$ to the differential two-form with values in $\mathbb{R}^q$ given by (\ref{eq:P2}), one obtains
\begin{multline}\label{eq:P3}
	0=\frac 12\xi^{\sigma}\xi^{\sigma'}\tau_{\sigma\sigma'}^s \lambda^{(\sigma')}\wedge\lambda^{(\sigma)}+\xi^{s}\xi^{\sigma}\lambda^{(\sigma)}\wedge \lambda_{,\gamma_{(\sigma)}}^{(s)}\\+\xi^\sigma\beta^s_{\sigma}\lambda^{(s)}\wedge\lambda^{(\sigma)}+ d\xi^s\wedge\lambda^{(s)}+\xi^sd_x\lambda^{(s)},\quad s=1,\ldots,k.
\end{multline}
The above expression is not summed over $s$, but it represents $k$ different equalities for $s=1,\ldots,k$.
Multiplying the above expression exteriorly  by $\lambda^{(s)}$, we obtain
\begin{multline}
	\label{eq:P5}
	0=\frac 12\xi^{\sigma}\xi^{\sigma'}\tau_{\sigma\sigma'}^s \lambda^{(\sigma')}\wedge\lambda^{(\sigma)}\wedge\lambda^{(s)}\\+\xi^s\xi^{\sigma'} \lambda^{(\sigma')}\wedge \lambda_{,\gamma_{(\sigma')}}^{(s)}\wedge \lambda^{(s)}
	+\xi^sd_x\lambda^{(s)}\wedge\lambda^{(s)},\qquad s=1,\ldots,k.
\end{multline}
Since the above expression must hold for any set of  $\xi^1,\ldots,\xi^k$, one obtains that 
\begin{equation}\label{Eq:Con0}
	\tau_{\sigma\sigma'}^s=0,\qquad  (\sigma-\sigma')(\sigma-s)(\sigma'-s)\neq 0.
\end{equation}
This is a very important result since, as shown in the Appendix, it allows us to rescale $\gamma_{({1})},\ldots,\gamma_{(k)}$ so that they commute among themselves and it leads to very important applications, as noted in Section \ref{Se:Appl}. 
Moreover,  (\ref{eq:P5}) gives that
\begin{equation}\label{Eq:Con2}
	0=\lambda^{(\sigma)}\wedge \lambda^{(s)}_{,\gamma_{(\sigma)}}\wedge \lambda^{(s)},\qquad \forall \sigma,s=1,\ldots,k.
\end{equation}
Additionally,
\begin{equation}\label{Eq:Con1}
	d_x\lambda^{(s)}\wedge\lambda^{(s)}=0,\qquad\forall s=1,\ldots,k.
\end{equation}
In view of the integrability condition for Pffafian systems, the forms $\lambda^{(s)}$ can be reparametrised by multiplying them by a function so  that  the new reparametrised differential one-forms, which we still denote by $\lambda^{(1)},\ldots,\lambda^{(k)}$,  will satisfy
\begin{equation}\label{Eq:Spe}
	\lambda^{(s)}=d_x\varphi^s(x,u),\qquad s=1,\ldots,k,
\end{equation}
for certain functions $\varphi^1(x,u),\ldots,\varphi^k(x,u)$, which are defined up to a function on $u$ for each index $s=1,\ldots,k$. As $x\in \mathbb{R}^p$, the functions $\varphi^1,\ldots,\varphi^k$ can be defined on the domain of $\lambda^{(1)},\ldots,\lambda^{(k)}$ as long as it is simply-connected as a consequence of the Poincar\'e Lemma. The functions $\varphi^1,\ldots,\varphi^k$ are called the {\it potential wave functions}. 
It is relevant to stress that every parametrisation (\ref{Eq:lsolution}) satisfying the previous conditions (\ref{Eq:Con0})--(\ref{Eq:Con1}) is such that the redefinition of each $\lambda^{(s)}$ due to  multiplication by a function so as to satisfy (\ref{Eq:Spe}) leads to a rescaling of the associated $\gamma_{(s)}$ by the same function so as to keep a parametrisation (\ref{Eq:lsolution}) of the same $k$-wave. Nevertheless, the new functions $\gamma_{(1)},\ldots,\gamma_{(k)}$ still satisfy (\ref{Eq:Con0})--(\ref{Eq:Con1}).

Consider the non-empty submanifolds of $\mathbb{R}^p$ of the form
$$
N_{\vec{a},u}:=\{x\in \mathbb{R}^{p}:\varphi^\sigma(x,u)=a^\sigma,\sigma=1,\ldots,k\},
$$
which depend on $u\in \mathbb{R}^q$ and $\vec{a}=(a^1,\ldots,a^k)\in \mathbb{R}^k$.
Using the expression (\ref{eq:gamma1}), we get
$$
d_x\gamma_{(\sigma)}|_{TN_{\vec{a},u}}=(\beta^{\sigma'}_{\sigma}\gamma_{(\sigma')}\otimes \lambda^{(\sigma')})|_{TN_{\vec{a},u}}=0,\qquad \sigma=1,\ldots,k.
$$
Therefore, each $\gamma_{(\sigma)}$ is a constant on each $N_{\vec{a},u}$ and 
$$
\gamma_{(\sigma)}=\gamma_{(\sigma)}(\varphi^1(x,u),\ldots,\varphi^{k}(x,u),u),\qquad \sigma=1,\ldots,k.
$$
Recall that the vector fields $\chi_1(x),\ldots,\chi_k(x)$ span a foliation on $\mathbb{R}^q$ of $k$-dimensional leaves for each value of $x$. Note that at points of a submanifold $N_{\vec{a},u}\times\{u\}\subset \mathbb{R}^{p+q}$, the vector fields $\chi_1,\ldots,\chi_k$ always span the same distribution on $\mathbb{R}^q$. 

Note that the condition (\ref{Eq:Con2}) implies that if  $s,\sigma=1,\ldots,k$ and $\sigma\neq s$, then $\lambda^{(s)}$ can be written as 
\begin{equation}\label{Eq:PLS}
	\lambda^{(s)}_{,\gamma_{(\sigma)}}=\zeta_{s1}\lambda^{(s)}+\zeta_{s2}\lambda^{(\sigma)}\Rightarrow \lambda^{(s)}_{,\gamma_{(\sigma)}}|_{TN_{\vec{a},u}}=0
\end{equation}
for some functions $\zeta_{s1},\zeta_{s2}\in C^\infty(\mathbb{R}^{p+q})$.

If  the condition (\ref{Eq:InvDis}) holds, then the distribution $\mathcal{D}_x$ is integrable and there exists a foliation of $\mathbb{R}^q$ into $k$-dimensional surfaces for every $x$.  If, in addition, (\ref{Eq:Con2}) and (\ref{Eq:Spe}) are satisfied, then one can consider an integral submanifold $\mathcal{S}\subset \mathbb{R}^q$ of $\mathcal{D}_x$ and consider a parametrisation, $f=f(\tau^1,\ldots,\tau^k)$, of $\mathcal{S}$ so that 
\begin{equation}	\label{eq:speX}
	u=f(\tau^1,\ldots,\tau^{k}),\qquad \tau^\sigma=\varphi^\sigma(x,f(\tau^1,\ldots,\tau^k)),\qquad \sigma=1,\ldots,k.	
\end{equation}
For each $x$, the tangent vectors to the leaves of $\mathcal{D}_x$ take values in the distribution $\mathcal{D}_x$. Hence, using (\ref{Eq:Con0}) the explicit parametrisation of the manifold $S$ requires us to solve the following system of PDEs
\begin{equation}
	\label{Eq:Sur}
	\frac{\partial f}{\partial \tau^\alpha}=\mu_\alpha^{\alpha'}\gamma_{(\alpha')},\qquad \alpha=1,\ldots,k,
\end{equation}
for certain functions $\mu_\alpha^{\alpha'}=\mu^{\alpha'}_\alpha(\tau^1,\ldots,\tau^k)$ for $\alpha,\alpha'=1,\ldots,k$. Moreover, one can write
$$
d_xu(x)=\frac{\partial f}{\partial \tau^\alpha}\otimes d_x\tau^\alpha.
$$
If $u(x)$ is a particular solution parametrised by $f(\tau^1,\ldots,\tau^k)$, then
$$
d_xu(x)=\mu_\alpha^{\alpha'}\gamma_{(\alpha')}\otimes d_x\tau^\alpha=\xi^{\alpha'}\gamma_{(\alpha')}\otimes \lambda^{(\alpha')}
$$and
\begin{equation}\label{Eq:star}
	d_x\tau^\alpha\in \langle \lambda^{(1)},\ldots,\lambda^{(k)}\rangle,\qquad \alpha=1,\ldots,k.
\end{equation}
Since $d_x\tau^1\wedge\ldots\wedge d_x\tau^k\neq 0 $ and $\lambda^{(1)}\wedge\ldots\wedge \lambda^{(k)}\neq 0$,  one obtains that $\lambda^{(1)},\ldots,\lambda^{(k)}$ are functions of the form
\begin{equation}\label{Eq:LambdaAlpha}
	\lambda^{(\alpha)}=\lambda^{(\alpha)}(x,f(\tau^1,\ldots,\tau^k)),\qquad \alpha=1,\ldots,k.
\end{equation}
%From this expression, it follows that a necessary condition for (\ref{}) to admit solutions is
The conditions (\ref{Eq:Con2}) and (\ref{Eq:Con1}), in parametric form, are
\begin{equation}
	\label{Eq:Con5}
	\frac{\partial \lambda^{(q)}}{\partial \tau^m}=\alpha_q^q\lambda^{(q)}+\alpha_m^{q}\lambda^{(m)},\quad \lambda^{(q)}=d_x\varphi^q(x,f(\tau^1,\ldots,\tau^k)),
\end{equation}
where $q,m=1,\ldots,k,$ with $q\neq m$, and the coefficients $\alpha^q_m$ depend on the variables $\tau^1,\ldots,\tau^k$. %
%Moreover, for every $u\in \mathbb{R}^q$, we have the vector fields on $\mathbb{R}^{p}$ of the form
%$$
%X_{(q)}=\gamma_{(q)}^i\frac{\partial}{\partial x^i},\qquad q=1,\ldots,s.
%$$

\begin{thm}\label{Th:primer} {\bf (A Riemann $k$-wave solution)} Let us assume the following conditions:
	\begin{enumerate}
		\item  $S$ is a $k$-dimensional surface in $\mathbb{R}^q$ parametrised by a function $f$ satisfying  (\ref{Eq:Sur}),
		\item  There exists a set of functions $\varphi^1(x,u),\ldots,\varphi^k(x,u)$ such that
		\begin{equation}
			\label{eq:0}
			\lambda^{(\alpha)}=d_x\varphi^\alpha(x,f(\tau^1,\ldots,\tau^k)),\qquad \alpha=1,\ldots,k,
		\end{equation}
		where the characteristic covectors $\lambda^{(1)},\ldots,\lambda^{(k)}$ satisfy the conditions (\ref{Eq:Con5}) and are related to the $k$-dimensional distribution spanned by the vector fields associated with $\gamma_{(1)},\ldots,\gamma_{(k)}$ through the wave relation (\ref{Rel:GammaDelta}),
		\item The differential $d_xf$ has the factorial form 
		\begin{equation}
			\label{eq:00}
			\frac{\partial f}{\partial \tau^\alpha}\otimes d_x\tau^\alpha=\gamma_{(\alpha)}\otimes \lambda^{(\alpha)}
		\end{equation}
		for the $k$ functionally independent differentials $d_x\tau^1,\ldots, d_x\tau^k$.
		
	\end{enumerate}
	If the set of implicitly defined relations between $u,x,\tau$ of the form
	\begin{equation}\label{Eq:Cases}
		\begin{cases}
			u=f(\tau^1,\ldots,\tau^k),\\
			\tau^\alpha=\varphi^\alpha(x,f(\tau^1,\ldots,\tau^k)),\qquad& \alpha=1,\ldots,k,
		\end{cases}
	\end{equation}
	can be solved on a certain open set of $\mathbb{R}^{p+q+k}$ so that $u$ and $\tau$ become functions of $x$, then  $u(x)=f(\tau^1(x)\ldots,\tau^k(x))$, given by (\ref{Eq:Cases}), constitutes a $k$-wave solution for the hyperbolic homogeneous system of PDEs (\ref{Eq:InPDE02}).
\end{thm}
\begin{proof}
	Using our assumptions, we obtain that
	$$
	d_xu(x)=\frac{\partial f}{\partial \tau^\alpha}\otimes d_x\tau^\alpha=\mu_\alpha^{\alpha'}\gamma_{(\alpha')}\otimes d_x\tau^\alpha=\gamma_{(\alpha)}\otimes \lambda^{(\alpha)}.
	$$	
	%Since $\lambda^{(1)}\wedge\ldots\wedge \lambda^{(k)}=\det(\mu^{\alpha'}_{\alpha})dt^1\wedge\ldots\wedge dt^k\neq0$, the matrix $\mu^{\alpha'}_\alpha$ is invertible.
	Consider the function $\Psi:\mathbb{R}^{p+q+k}\rightarrow \mathbb{R}^{q+k}$ given by
	$$
	\Psi^\sigma(x,\tau,u)=\tau^\sigma-\varphi^\sigma(x,u),\qquad \sigma=1,\ldots,k,
	$$
	$$
	\Psi^{\beta+k}(x,\tau,u)=u^\beta-f^\beta(\tau^1,\ldots,\tau^k),\qquad \beta=1,\ldots,q.
	$$
	Then, the condition $\Psi(x,\tau,u)=0$ allows us to determine $\tau=\tau(x),u=u(x)$ via the Implicit Function Theorem, when  the determinant of the $(q+k)\times (q+k)$ matrix
	$$
	A:=\left[\begin{array}{c|c}
		{\rm Id}_{q\times q}&-\frac{\partial f^{\alpha}}{\partial \tau^\sigma}\\
		\hline
		-\frac{\partial \varphi^{\sigma}}{\partial u^\alpha}&{\rm Id}_{k\times k}
	\end{array}\right]
	$$
	is assumed to be different from zero. By multiplying $A$ by 
	$$
	\left[\begin{array}{c|c}
		{\rm Id}_{q\times q}&\frac{\partial f^\alpha}{\partial \tau^{\sigma'}}\\
		\hline
		0&{\rm Id}_{k\times k}
	\end{array}\right]
	$$
	on the left, we obtain that 
	$$
	\det A=\det \left[\begin{array}{c|c}
		{\rm Id}_{q\times q}-\frac{\partial \varphi^\sigma}{\partial u^\alpha}\frac{\partial f^\alpha}{\partial \tau^{\sigma'}}&0\\
		\hline
		-\frac{\partial \varphi^{\sigma'}}{\partial u^\sigma}&{\rm Id}_{k\times k}
	\end{array}\right],$$
	where we assume 
	$$
	\det\left({\rm Id}_{q\times q}-\frac{\partial \varphi^\sigma}{\partial u^\alpha}\frac{\partial f^\alpha}{\partial \tau^\sigma}\right)\neq 0
	$$
	to exclude the gradient catastrophe. Recall that from the second relation in (\ref{Eq:Cases}), one has that
	$$
	d_x\tau^{\sigma'}=\left[\delta^{\sigma'}_\sigma-\frac{\partial \varphi^\sigma}{\partial u^\alpha}\frac{\partial f^\alpha}{\partial \tau^{\sigma'}}\right]\lambda^{(\sigma)},
	$$
	where $d_x\tau^1\wedge\ldots\wedge d_x\tau^q\neq 0$ because of assumption 3. 
	Hence, $\det A\neq 0$ since  $\lambda^{(1)}\wedge \ldots\wedge \lambda^{(k)}\neq 0$.  Using the definition of characteristic vectors, we obtain that $f(\tau^1(x),\ldots,\tau^k(x))$ satisfies equation (\ref{Eq:InPDE02}), which completes the proof.
\end{proof}

Surfaces of the form $u=f(\tau^1,\ldots,\tau^k)$ are called {\it hodograph surfaces} immersed in $\mathbb{R}^q$ and the function $f(\tau^1,\ldots,\tau^k)$ is called a {\it hodograph function}. When the conditions (\ref{Eq:Sur}) and (\ref{Eq:Con5}) are also satisfied, then we say that the surface is {\it involutive}. The functions $\tau^1,\ldots,\tau^k$ are called {\it generalised Riemann invariants}.

Let us prove a relevant consequence of the existence of solutions of (\ref{Eq:lsolution}), understood as a first-order system of partial differential equations in normal form for $u(x)$. As $u(x)=f(\varphi(x,u))$, where $\varphi(x,u)=(\varphi^1(x,u),\ldots,\varphi^k(x,u))$, one has that
$$
d_xu^\gamma(x)=\frac{\partial f^\gamma}{\partial \varphi^a}d_x\varphi^a+\frac{\partial \varphi^a}{\partial u^\beta}d_xu^\beta(x)\frac{\partial f^\gamma}{\partial \varphi^a}.
$$
Then
$$
d_xu^\gamma(x)=(\Psi^{-1})^\gamma_\beta\frac{\partial f^\beta}{\partial \varphi^j}d_x\varphi^j,
$$
where 
$$
\Psi^\alpha_\beta=\delta^\alpha_\beta-\frac{\partial \varphi^a}{\partial u^\beta}\frac{\partial f^\alpha}{\partial \varphi^a}.
$$
If $\langle\lambda^{(a)}(x,u),\vartheta(x,u)\rangle=0$ for every $a=1,\ldots,k$, then
$$
\langle d_xu^\gamma(x),\vartheta(x,u)\rangle=(\Psi^{-1})^\gamma_\beta\frac{\partial f^\beta}{\partial \varphi^i}\langle d_x\varphi^i,\vartheta\rangle=0. 
$$
Hence, according to (\ref{Eq:LambdaAlpha}), the directions in $\bigcap_{a=1}^k \ker \lambda^{(a)}_x$ give curves along which $u(x)$ takes a constant value. 

From the proof of Theorem \ref{Th:primer}, we obtain
$$
\mu^{\alpha'}_{\alpha}d\tau^\alpha=\lambda^{(\alpha')},\qquad \forall \alpha'=1,\ldots,k.
$$
Hence, $\lambda^{(\sigma)}=\lambda^{(\sigma)}(x,f(\tau^1,\ldots,\tau^k))$ for $\alpha=1,\ldots,k$. From the integrability conditions (\ref{Eq:Con2}), it follows that
$$
\frac{\partial \lambda^{(\sigma)}}{\partial \tau^{\sigma'}}=g^{\sigma}_{\sigma'}\lambda^{(\sigma)}+g^{\sigma'}_{\sigma}\lambda^{(\sigma')},\qquad \sigma\neq \sigma',
$$
are conditions together with (\ref{Eq:Con5}) for the existence of %Riemann
$k$-waves solutions of (\ref{Eq:InPDE02}). % of the system (\ref{Eq:InPDE02}). %One can see that they are also sufficient conditions.

\section{A geometric approach to hydrodynamic equations}\label{Se:GA}

Let us provide a geometric interpretation of the hydrodynamic equations (\ref{Eq:InPDE02}) for which %in order to simplify the notation will be written 
%\begin{equation}\label{Eq:HydroAgain}
%\sum_{i=1}^{p }\sum_{\beta=1}^q  
%A^{i\alpha}_\beta(x, {u}) {u}^\beta_i=0,\qquad \alpha=1,\ldots,q
%\end{equation}
we assume that the %and a way to represent its $k$-wave solutions, i.e. we assume that 
solutions satisfy 
\begin{equation}\label{Eq:FirstTu}
	T_xu=%\sum_{a=1}^k
	\xi^a(x)\gamma_{(a)}(x,u)\otimes \lambda^{(a)}(x,u).
\end{equation}
It is worth recalling that $T_xu$  stands for the tangent map to $u:x\in \mathbb{R}^p\mapsto u(x)\in \mathbb{R}^q$ at $x\in \mathbb{R}^p$. Recall again that, in the literature about hydrodynamic equations, $T_xu$ is more commonly denoted by $du$. Unfortunately, $du$ is misleading from a geometric point of view, as $d$ can be understood as a generalisation of the exterior differential on a manifold whose meaning need not be unique or even properly defined on a map $u:\mathbb{R}^p\rightarrow \mathbb{R}^q$.

Substituting (\ref{Eq:FirstTu})  into (\ref{Eq:InPDE02}), skipping the dependence on independent and/or dependent variables, and taking into account our previous remarks, we obtain
\begin{equation}\label{Eq:Hyd}
	%\sum_{a=1}^k\sum_{i=1}^{p  }\sum_{\beta=1}^q  
	A^{i\alpha}_\beta\xi^a\gamma_{(a)}^\beta\lambda^{(a)}_i=0,\qquad \alpha=1,\ldots,q.
\end{equation}
Since we want our solutions to remain valid for every value of the functions $\xi^1,\ldots,\xi^k$ at a point, it follows that
\begin{equation}\label{Eq:Hyd2}
	%\sum_{a=1}^k\sum_{i=1}^{p  }\sum_{\beta=1}^q  
	A^{i\alpha}_\beta\gamma_{(a)}^\beta\lambda^{(a)}_i=0,\qquad \alpha=1,\ldots,q,\quad a=1,\ldots,k.
\end{equation}

Recall that the $\lambda_i^{(a)}$, with $i=1,\ldots,p$, are the coordinates of each $\lambda^{(a)}$ while $\gamma_{(a)}^\beta$, for $\beta=1,\ldots,q$, are the coordinates of each $\gamma_{(a)}$. Let us rewrite (\ref{Eq:Hyd}) more geometrically by defining some new associated geometric structures. First, let us define an $\mathbb{R}^q$-parametrised family of differential one-forms on $\mathbb{R}^p$ taking values in $\mathbb{R}^k$ of the form
\begin{equation}%\label{Eq:lambda}
	\lambda\!\!\!\!\lambda=%\sum_{i=1}^{p }\sum_{a=1}^k
	\lambda^{(a)}_idx^i\otimes e_{(a)},
\end{equation}
where $\{e_{(1)},\ldots,e_{(k)}\}$ is a basis of $\mathbb{R}^k$. The space of such differential forms can be denoted by $C^\infty(\mathbb{R}^q)\otimes \Omega^1(\mathbb{R}^p)\otimes \mathbb{R}^k$. 
Additionally, we define a family $ \mathfrak{X}\!\!\!\!\mathfrak{X}_a^1,\ldots,\mathfrak{X}\!\!\!\!\mathfrak{X}_a^q$, with $a=1,\ldots,k$, of $\mathbb{R}^q$-parametrised vector fields on $\mathbb{R}^p$ taking values in $(\mathbb{R}^{k})^*$ given by
\begin{equation}\label{Eq:Xaa}
	\mathfrak{X}\!\!\!\!\mathfrak{X}_a^\alpha=%\sum_{i=1}^{p }
	A^{i\alpha}_\beta\gamma_{(a)}^\beta \frac{\partial}{\partial x^i}\otimes e^{(a)},\qquad \alpha=1,\ldots,q,\quad a=1,\ldots,k,
\end{equation}
where $\{e^{(1)},\ldots,e^{(k)}\}$ stands for the dual basis of $\{e_{(1)},\ldots,e_{(k)}\}$. It is worth noting that the above expression is not summed over the values of the subindex $a$.
Then, it makes sense to define the contraction of $\lambda\!\!\!\lambda$ and $\xi^a\mathfrak{X}\!\!\!\!\mathfrak{X}_a^\alpha$, where the sum is over $a=1,\ldots,k$, which makes the hydrodynamic equations  look like 
\begin{equation}\label{Eq:contrac}
	\left\langle \lambda\!\!\!\!\lambda,%\sum_{a=1}^{k }
	\xi^a\mathfrak{X}\!\!\!\!\mathfrak{X}_a^\alpha\right\rangle=0,\qquad \forall \alpha=1,\ldots,q.
\end{equation}
Since we expect our solution to be valid for all possible values of $\xi^1,\ldots,\xi^k$ at an arbitrary point, it follows that
\begin{equation}\label{Eq:Hydro}
	\langle \lambda\!\!\!\!\lambda,\mathfrak{X}\!\!\!\!\mathfrak{X}_a^\alpha\rangle=0,\qquad \forall \alpha=1,\ldots,q,\quad \forall a=1,\ldots,k.
\end{equation}
%Let us define the following $k$-vector field on $\mathbb{R}^p$ taking values in $\mathbb{R}^{k*}$ parametrised by elements of $\mathbb{R}^q$ of the form 
%$$
%\mathfrak{X}\!\!\!\!\mathfrak{X}^\alpha=\sum_{a=1}^k\sum_{i=1}^{p }
%A^{i\alpha}_\beta\gamma_{(a)}^\beta \frac{\partial}{\partial x^i}\otimes e^{(a)},\qquad \alpha=1,\ldots,q,%\quad a=1,\ldots,k,
%$$
%Hence, equation (\ref{Eq:Hydro}) becomes
%\begin{equation}\label{Eq:Con}
%\langle \lambda\!\!\!\!\lambda,\mathfrak%{X}\!\!\!\!\mathfrak{X}^\alpha\rangle=0,\qquad \forall \alpha=1,\ldots,q.
%\end{equation}
We are especially interested in the case when $ \lambda\!\!\!\!\lambda$ is an exact differential form, i.e. there exists a $\mathbb{R}^{q}$-parametrised zero-form $\Upsilon\!\!\!\!\!\Upsilon$ on $\mathbb{R}^p$ taking values in $\mathbb{R}^k$ such that
\begin{equation}\label{eq:gammaUpsilon}
	\lambda\!\!\!\!\lambda=d\Upsilon\!\!\!\!\!\Upsilon.
\end{equation}
Let us propose a method for obtaining $ \lambda\!\!\!\!\lambda$ while we ensure that a $k$-wave solution exists. Consider the distribution on $\mathbb{R}^p$, parametrised by elements of $\mathbb{R}^q$, spanned by $\mathfrak{X}\!\!\!\!\mathfrak{X}_a^1,\ldots,\mathfrak{X}\!\!\!\!\mathfrak{X}_a^q$ for $a=1,\ldots,k$. % Indeed, note that the $k$-vector fields $\mathfrak{X}^\alpha\!\!\!\!\!\!\!\!\mathfrak{X}^\alpha=\mathfrak{X}_1^\alpha\!\!\!\!\!\!\!\!\mathfrak{X}_1^\alpha+\ldots+\mathfrak{X}_k^\alpha\!\!\!\!\!\!\!\!\mathfrak{X}_k^\alpha$ induces an associated distribution on $\mathbb{R}^p$ for every $u\in \mathbb{R}^q$ spanned by their components. The $k$-vector field $\mathfrak{X}\!\!\!\!\mathfrak{X}$ is said to be integrable when its components span an integrable distribution. 
Assume that $\gamma_{(1)},\ldots,\gamma_{(k)}$ are chosen so that the $\mathfrak{X}\!\!\!\!\mathfrak{X}_a^\alpha$, for $\alpha=1,\ldots,q$ and $a=1,\ldots,k$, span an involutive distribution on $\mathbb{R}^p$ for every value $u\in \mathbb{R}^q$. In this case, we can choose, for every $u\in \mathbb{R}^q$, certain $\mathbb{R}^q$-dependent functions $\Upsilon_1,\ldots,\Upsilon_k$ on $\mathbb{R}^p$ that are constant on the leaves of the involutive distribution (at least locally at a generic point). Hence, (\ref{Eq:Hydro}) follows. This has to be done in such a way that the $\gamma_{(1)},\ldots\gamma_{(k)}$ span an integral distribution on $\mathbb{R}^q$ for every point in $\mathbb{R}^p$. We recall that a $k$-vector field is said to be integrable when its components span an integrable distribution. Hence, one can say that we are interested in an integrable $k$-vector field on $\mathbb{R}^q$  of the form
\begin{equation}\label{Eq:gammae}
	\gamma\!\!\!\gamma=
	%\sum_{a=1}^k\
	\gamma_{(a)}\otimes e^{(a)},
\end{equation}
for every fixed value of $x\in \mathbb{R}^p$. Moreover, we expect that, in this  case, we can consider  for every value of $u\in\mathbb{R}^q$  the leaves of the induced distribution on $\mathbb{R}^p$ spanned by $\mathfrak{X}\!\!\!\!\mathfrak{X}_a^1,\ldots,\mathfrak{X}\!\!\!\!\mathfrak{X}_a^q$ for $a=1,\ldots,k$. Hence, we can define a family of functions vanishing on them. This ensures that $d\Upsilon\!\!\!\!\!\Upsilon$ vanishes on the $\mathfrak{X}\!\!\!\!\mathfrak{X}^1_a,\ldots,\mathfrak{X}\!\!\!\!\mathfrak{X}^q_a$ with $a=1,\ldots,k$. This is to be done in such a way that the vector fields $\gamma\!\!\!\!\gamma$ span an involutive distribution on $\mathbb{R}^q$ for every fixed $x$. %This ensures that they will provide a solution to our problem.

\begin{prop}
	Let $\mathbb{R}^k\ni\tau\mapsto f(\tau)\in \mathbb{R}^q$ be an immersed submanifold $S$ in $\mathbb{R}^q$. Consider the  $\mathbb{R}^q$-parametrised vector fields on $\mathbb{R}^p$ taking values in $(\mathbb{R}^k)^*$ of the form 
	$$
	\mathfrak{X}\!\!\!\!\mathfrak{X}_a^\alpha = A_{\beta}^{i\alpha}\gamma_{(a)}^\beta\frac{\partial}{\partial x^i}\otimes e^{(a)},\qquad \alpha=1,\ldots,q,\qquad a=1,\ldots,k,
	$$
	where the $A^{i\alpha}_\beta$ appear in (\ref{Eq:InPDE02}). Suppose that there exists a $\mathbb{R}^q$-parametrised zero-form $\Upsilon\!\!\!\!\!\Upsilon=\sum_{a=1}^kT^a\otimes e_{(a)}$ on $\mathbb{R}^p$ taking values in  $\mathbb{R}^k$ such that $\lambda\!\!\!\!\lambda =d\Upsilon\!\!\!\!\!\Upsilon$ and assume that the $\gamma_{(a)}=\gamma_{(a)}(x,u)$ take a constant value on the submanifolds $N_{\vec{a},u}$. Assume also that each $\partial f/\partial T^s$ is proportional to $\gamma_{(s)}$ and the conditions (\ref{Eq:Con2}) are satisfied. Then,  $u(x)=f(T^1(x),\ldots,T^k(x))$ constitutes an exact Riemann $k$-wave solution of the system (\ref{Eq:lsolution}). 
\end{prop}	

Note that the $T^s$ are first-integrals of the vector fields $\mathfrak{X}\!\!\!\!\mathfrak{X}^\alpha _a$. As shown in   the examples of the next section, the existence of the functions $T^s$ imply, in some cases, that the $\mathfrak{X}\!\!\!\!\mathfrak{X}^\alpha _a$ span an involutive distribution.

\section{Examples}\label{Se:Appl}

Now we shall present a few examples which illustrate our theoretical considerations. 

$\bullet$ Let us return to the Brownian motion equation (\ref{Ex:Examples}) with $\kappa=1$. In this case, the matrix related to $T_xu$ is given by
$$
[T_xu]=\left[ \begin{matrix} \frac {\partial \widetilde{a}} {\partial t}& \frac {\partial \widetilde{a}} {\partial x}& \frac {\partial \widetilde{a}} {\partial y}\\ \frac {\partial \widetilde {c}} {\partial t}& \frac {\partial \widetilde {c}} {\partial x}& \frac {\partial \widetilde{c}} {\partial y}\end{matrix} \right] =\xi \left[ \begin{matrix} \gamma^{1}\\ \gamma^{2}\end{matrix} \right] \otimes \left[ \begin{matrix} \lambda _{1}& \lambda _{2}& \lambda _{3}\end{matrix} \right] =\xi \left[ \begin{matrix} \gamma^{1}\lambda _{1}& \gamma ^{1}\lambda _{2}& \gamma^{1}\lambda _{3}\\ \gamma ^{2}\lambda _{1}& \gamma ^{2}\lambda_{2}& \gamma^{2}\lambda _{3}\end{matrix} \right],$$
where we recall that $a=u_t$ and $c=u_x$. 
Meanwhile, we are interested in solving the  equation
$$
\lambda_{1}\left[\begin{array}{cc} 0& 0\\ 
	0& -1\end{array} \right] \left[ \begin{array}{c} \gamma^1\\ \gamma^2\end{array} \right] _{t}+\lambda _{2}\left[ \begin{array}{cc} 0&\frac {1+\beta^{2}x^{2}} {\widetilde{a}+y}\\1& 0 \end{array} \right] \left[\begin{array}{c} \gamma^1\\ \gamma^2\end{array}\right] _{x}+\lambda_3\left[ \begin{matrix} 1& 0\\ 0& 1\end{matrix} \right] \left[ \begin{array}{c} \gamma^1\\ \gamma^2\end{array} \right]_y \!\!=\!\!\left[ \begin{matrix} 0\\ 0\end{matrix} \right].
$$
and we therefore analyse the vector fields
$$
X_1=\frac{1+\beta^2x^2}{\widetilde{a}+y}\gamma^2\frac{\partial}{\partial x}+\gamma^1\frac{\partial}{\partial y},\qquad X_2=-\gamma^2\frac{\partial}{\partial t}+\gamma^1\frac{\partial}{\partial x}+\gamma^2\frac{\partial}{\partial y}.
$$
Note that, indeed, if we want $\varphi(t,x,y,u)$ to exist, the vector fields $X_1,X_2$ must be such that their Lie bracket belongs to the distribution on $\mathbb{R}^3$ spanned, at every $u\in \mathbb{R}^2$ by $X_1,X_2$. Consider then
$$
[X_1,X_2]=-X_1\gamma^2\frac{\partial}{\partial t}+\left(X_1\gamma^1-X_2\frac{1+\beta^2x^2}{a}\right)\frac{\partial}{\partial x}+(X_1\gamma^2-X_2\gamma^1)\frac{\partial}{\partial y}
$$
If $\gamma^2=0$ and $\gamma^1\neq 0$, then the distribution generated by $X_1,X_2$ for every value of $\widetilde{a},\widetilde{c}$ is involutive as well as the distribution spanned by $\gamma$. Indeed, the obtained distribution reads
$$
\left\langle \frac{\partial}{\partial x},\frac{\partial}{\partial y}\right\rangle\Rightarrow \Upsilon=t.
$$
Hence, $\lambda=d_{(t,x,y)}t=dt$.  Recall that $\gamma^1,\gamma^2$ must be constant on the submanifolds where $\Upsilon$ takes a constant value. Then, $\gamma^1=\gamma^1(t)$. Since $\gamma$ is a unique vector, we do not have any additional integrability condition. Solutions can be obtained by setting
$$
u=\left(\int \gamma^1(t)dt,\widetilde{c}_0\right)^T.
$$

It is worth noting that this is not the only possible choice for $\Upsilon$. 
Hence, we obtain a simple wave. 

$\bullet$ Let us analyse a simple hydrodynamical type equation given by 
\begin{equation}\label{Eq:Example3}
	\frac{\partial u}{\partial t}+xu\frac{\partial u}{\partial x}+yu^2\frac{\partial u}{\partial y}=0,
\end{equation}
whose solutions are of the form $u:\mathbb{R}^3\rightarrow\mathbb{R}$.
As previously, let us parametrise the tangent map to the solution $u(p)$ in the form
$$
d_pu(p)=\gamma(t,x,y,u)\otimes \lambda(t,x,y,u)
$$
to look for simple waves of (\ref{Eq:Example3}). Then, $\lambda=(\lambda_1,\lambda_2,\lambda_3)$ and its associated $\gamma$ satisfy the relation
\begin{equation}\label{Eq:DisRe}
	(\lambda_1+xu\lambda_2+yu^2\lambda_3)\gamma=0.
\end{equation}
Let us determine the values of $\lambda =\lambda_1dt+\lambda_2dx+\lambda_3dy$ that give a solution of (\ref{Eq:DisRe}) for any arbitrary associated $\gamma$. In this case,
$$
\lambda_1=-xu\lambda_2-yu^2\lambda_3.
$$
Hence, $\lambda$ takes the form
$$
\lambda=(-xudt+dx)\lambda_2+(-yu^2dt+dy)\lambda_3
$$
for arbitrary values of $\lambda_2,\lambda_3$. To apply our techniques, we require $d_p\lambda=0$. In particular, let us assume $\lambda_2=m/x$ and $\lambda_3=k/y$ for any pair of constants $m,k\in \mathbb{R}$ such that $m^2+k^2\neq 0$. Hence,
$$
\lambda=\left(-udt+\frac{dx}{x}\right)m+\left(-u^2dt+\frac{dy}{y}\right)k
$$
satisfies $d_p\lambda=0$ at $p=(t,x,y)$. Hence, we can ensure
$
\lambda=d_pR
$
for
$$
R=-t(um+u^2k)+\ln|x^my^k|.
$$
Assuming that $\gamma$ is a constant on the submanifolds of $\mathbb{R}^3$, where $R$ takes a constant value, a solution of (\ref{Eq:Example3})  can be written as $u=f(R)$, where $f(R)$ is a solution of 
\begin{equation}\label{Eq:Conu}
	\frac{df}{dR}=\gamma(R,u).
\end{equation}
Since $\gamma$ is an arbitrary function of the form $\gamma(R,u)$, a solution can be written as
$$
u=f(R)
$$
for an appropriate function $\gamma(R,u)$. With respect to $R$, let us recall that equation (\ref{Eq:eqeq}) can be determined in an implicit manner by
the equation 
$$
R=-t(f(R)m+f(R)^2k)+\ln|x^my^k|
$$
for a solution $f$ of (\ref{Eq:Conu}). In other words, the above expression determines in an implicit manner the value of $R$ as a function of $t,x,y$.  Then, in an implicit manner, 
$$
u=f(-t(um+u^2k)+\ln|x^my^k|)
$$
describes the solutions of (\ref{Eq:Example3}) determined by generalised Riemann invariants. Let us provide a more particular case. Consider $\gamma=c$, which takes a constant value on the submanifolds of $\mathbb{R}^3$, where $R$ takes a constant value. Hence, $f(R)=c R$ and 
$$
R=-t(c Rm+c^2R^2k)+\ln|x^my^k|,
$$
which, assuming that $k^2+c^2\neq 0$,  gives that
$$R=-\frac{\sqrt{4 c ^2 k t \log \left(x^m y^k\right)+(cm    t+1)^2}+c  tm+1}{2 c^2  k t},
$$
and
$$ u=f(R(t,x,y))=-\frac{\sqrt{4 c^2 k t \log \left(x^m y^k\right)+(c tm+1)^2}+c tm+1}{2 c k t},$$
which can be proved to be an explicit  solution of (\ref{Eq:Example3}). 

$\bullet$ Let us analyse the hydrodynamic system of partial differential equations with two dependent and three independent variables of the form
\begin{equation}\label{Eq:Example2}
	\begin{cases}
		\frac{\partial u_1}{\partial t}+x\frac{\partial u_2}{\partial x}+yu_1\frac{\partial u_2}{\partial y}=0,\\
		\frac{\partial u_2}{\partial t}+xu_2\frac{\partial u_1}{\partial x}+y\frac{\partial u_1}{\partial y}=0.\\
	\end{cases}
\end{equation}
We have the parametrization
$$
[T_pu]=\left[ \begin{matrix} \frac {\partial u_1} {\partial t}& \frac {\partial u_1} {\partial x}& \frac {\partial u_1} {\partial y}\\ \frac {\partial u_2} {\partial t}& \frac {\partial u_1} {\partial x}& \frac {\partial u_1} {\partial y}\end{matrix} \right] =\xi \left[ \begin{matrix} \gamma^{1}\\ \gamma^{2}\end{matrix} \right] \otimes \left[ \begin{matrix} \lambda _{1}& \lambda _{2}& \lambda _{3}\end{matrix} \right] =\xi \left[ \begin{matrix} \gamma^{1}\lambda _{1}& \gamma ^{1}\lambda _{2}& \gamma^{1}\lambda _{3}\\ \gamma ^{2}\lambda _{1}& \gamma ^{2}\lambda_{2}& \gamma^{2}\lambda _{3}\end{matrix} \right].$$
The hydrodynamic equation (\ref{Eq:Example2}) in matrix form can be written as
$$
\lambda_{1}\left[\begin{array}{cc} 1& 0\\ 
	0& 1\end{array} \right] \left[ \begin{array}{c} \gamma^1\\ \gamma^2\end{array} \right] _{t}+\lambda _{2}\left[ \begin{array}{cc} 0&x\\u_2x& 0 \end{array} \right] \left[\begin{array}{c} \gamma^1\\ \gamma^2\end{array}\right] _{x}+\lambda_3\left[ \begin{matrix} 0& u_1y\\ y& 0\end{matrix} \right] \left[ \begin{array}{c} \gamma^1\\ \gamma^2\end{array} \right]_y =\left[ \begin{matrix} 0\\ 0\end{matrix} \right],
$$
where
$$
\gamma=\gamma^1\frac{\partial}{\partial u_1}+\gamma^2\frac{\partial}{\partial u_2}.
$$
Here, consider the problem of the construction of a Riemann double  wave. Let us consider the vector fields (\ref{Eq:Xaa}) for (\ref{Eq:Example2}) of the form $\mathfrak{X}\!\!\!\!\mathfrak{X}^1=X^1\otimes e^{(1)}$ and $\mathfrak{X}\!\!\!\!\mathfrak{X}^2=X^2\otimes e^{(1)}$, where
$$
X_1=\gamma^1\frac{\partial}{\partial t}+x\gamma^2\frac{\partial}{\partial x}+yu_1\gamma^2\frac{\partial}{\partial y},\qquad X_2=\gamma^2\frac{\partial}{\partial t}+xu_2 \gamma^1\frac{\partial}{\partial x}+y\gamma^1\frac{\partial}{\partial y}.
$$
More conveniently, let us define $\Delta=\gamma^1/\gamma^2$ and we verify that the vector fields
$$
Y_1=\Delta\frac{\partial}{\partial t}+x\frac{\partial}{\partial x}+yu_1\frac{\partial}{\partial y},\qquad Y_2=\frac{\partial}{\partial t}+xu_2\Delta\frac{\partial}{\partial x}+y\Delta\frac{\partial}{\partial y},
$$
span an involutive distribution of rank two, which implies that they have a common constant of the motion. Note that 
$$
[Y_1,Y_2]=-(Y_2\Delta)\frac{\partial}{\partial t}+(Y_1(u_2x\Delta)-Y_2x)\frac{\partial}{\partial x}+[Y_1(y\Delta)-Y_2(u_1y)]\frac{\partial}{\partial y}.
$$
Hence,
$$
[Y_1,Y_2]=-(Y_2\Delta)\frac{\partial}{\partial t}+(Y_1(xu_2\Delta)-Y_2x)\frac{\partial}{\partial x}+y(Y_1\Delta)]\frac{\partial}{\partial y}.
$$Consider $u_1>0$. This implies that if $\Delta=\pm\sqrt{u_1}$, we obtain that
$$
(u_1u_2-1)x\frac{\partial}{\partial x}=\pm\sqrt{u_1}Y_2-Y_1
$$
is proportional to $[Y_1,Y_2]$. Note also that such values of $\Delta_\pm$ allow us to define $\gamma^1,\gamma^2$ to be constant. Hence, they will be constant on the regions where $\tau_\pm=\tau_\pm(t,x,y)$ take a constant value, which allows us to satisfy condition (\ref{eq:gamma1}). Hence, a common constant of the motion for $Y_1,Y_2$ can be found. In particular, it is immediate that the functions corresponding to the relation (\ref{eq:gammaUpsilon}) become
$$
\Upsilon^\pm=t\mp\frac{\ln|y|}{\sqrt{u_1}}
$$
and are constants of the motion. Therefore, again by using (\ref{eq:gammaUpsilon}), we define
$$
\lambda^\pm=d_x\Upsilon^\pm=dt\mp\frac{dy}{y\sqrt{u_1}}.
$$
To find a simple wave, we consider a characteristic of the $\mathbb{R}^3$-parametrised vector field on $\mathbb{R}^2$ of the form
$$
\gamma_\pm=\gamma_\pm^1\frac{\partial}{\partial u_1}+\gamma_\pm^2\frac{\partial}{\partial u_2}
$$
Since $\Delta_\pm=\pm \sqrt{u_1}=\gamma^1_\pm/\gamma_\pm^2$, we choose $\gamma_\pm^1=\pm\sqrt{u_1}$ and $\gamma_\pm^2=1$. Then, we look for a solution of the characteristic system
$$
\frac{du_1}{d\tau}=\pm\sqrt{u_1},\qquad \frac{du_2}{d\tau}=1,
$$
which reads
$$
u_{1\pm}(\tau)=\left(\frac{\pm\tau_\pm-\tau_0}2\right)^2,\qquad u_{2\pm}(\tau)=\tau_\pm-\tau_0',
$$
for arbitrary $\tau_0,\tau_0'\in \mathbb{R}$.
The functions $\tau_\pm$ can be written as  functions of $t,x,y$ by solving the implicit system
$$
f_\pm(\tau_\pm)=(u_{1\pm}(\tau),u_{2\pm}(\tau)),\qquad \tau_\pm=\Upsilon_\pm(t,x,y,u_1,u_2)
$$
which give, for instance, the following solutions
$$
\tau_{\pm}=\frac{t\pm \tau_0-\sqrt{(t\mp\tau_0)^2-8\ln|y|}}{2}.
$$
Consequently, we obtain the explicit solutions of our original system of PDEs (\ref{Eq:Example2}) given by
$$
u_{1\pm}(t,x,y)=\frac{1}{4}\left(\pm\frac{ t\pm\tau_0-\sqrt{( t\mp\tau_0)^2-8\ln|y|}}{2}-\tau_0\right)^2,
$$
$$u_{2\pm}(t,x,y)=\frac{ t\pm\tau_0-\sqrt{( t\mp\tau_0)^2-8\ln|y|}}{2}-\tau_0',
$$
for arbitrary $\tau_0,\tau_0'\in \mathbb{R}$. 
It is worth remarking that, if $p$ stands for the variables in $\mathbb{R}^3$, then
\begin{multline*}
	d_{p}u_\pm(t,x,y)\!\!=\!\!\left[\begin{array}{c}\!\!\pm\frac{du_1}{d\tau^\pm}(\tau^\pm(t,x,y))\\\!\!\frac{du_2}{d\tau^\pm}(\tau^\pm(t,x,y))\end{array}\right]\otimes d_p\tau^\pm(t,x,y)\\=\!\left[\begin{array}{c}\!\!\pm\sqrt{u_1(t,x,y)}\\1\end{array}\right]\otimes d_p\tau^\pm(t,x,y),
\end{multline*}
where $p=(t,x,y)$. But then,
$$
{\small d_p\tau^\pm\!=\!\frac{\sqrt{(t\mp \tau_0)^2-8\ln |y|}-t\pm\tau_0}{2\sqrt{( t\mp t_0)^2-8\ln|y|}}\left(dt\!-\!\frac{4dy}{y( t\mp \tau_0-\sqrt{(t\mp\tau_0)^2-8\ln |y|})}\right)}
$$
and hence
$$
d_p \tau_\pm=\frac{\sqrt{(t\mp \tau_0)^2-8\ln |y|}-t\pm\tau_0}{2\sqrt{( t\mp t_0)^2-8\ln|y|}}\lambda^\pm(t,x,y,u_{1\pm}(t,x,y),u_{2\pm}(t,x,y)),
$$
as established in the proof of Proposition \ref{Eq:eqeq}. Therefore, $u_\pm(t,x,y)$ are simple wave solutions for (\ref{Eq:Example2}).  

Finally, observe that $\gamma_{\pm}=\pm \sqrt{u_1}\partial_{u_1}+\partial_{u_2},\lambda^{\pm}$ are such that 
$$
[\gamma^+,\gamma^-]=0,\qquad \gamma^+\wedge\gamma^-\neq 0,\qquad \lambda^+\wedge\lambda^-\neq 0,
$$
which allows us to obtain a superposition of simple waves giving rise to a double-wave solution with
$$
T_pu(t,x,y)=\xi^+\gamma_+\otimes \lambda^++\xi^-\gamma_-\otimes \lambda^-,
$$
for certain coefficients $\xi^+,\xi^-:\mathbb{R}^3\rightarrow \mathbb{R}$. 
To obtain a $2$-wave, we are going to consider the system of partial differential equations
$$
\frac{\partial u_1}{\partial \tau^+}=\frac 12\sqrt{u_1},\qquad \frac{\partial u_1}{\partial \tau^-}=-\frac 12\sqrt{u_1},\qquad \frac{\partial u_2}{\partial \tau^+}=\frac 12,\qquad \frac{\partial u_2}{\partial \tau^-}=\frac 12.
$$
This is indeed the coordinate representation of the system
$$
\frac{\partial u}{\partial \tau^+}=\frac 12\gamma_+,\qquad \frac{\partial u}{\partial\tau_-}=\frac 12\gamma_-.
$$
It is simple to see that a possible solution is given by
$$
u_1=\left(\frac{\tau^+-\tau^-}{4}\right)^2,\qquad u_2=\frac 12(\tau^++\tau^-).
$$
This allows us to define a mapping $f:\mathbb{R}^2\rightarrow \mathbb{R}^2$ of the form
$$
f(\tau^+,\tau^-)=\left(\left[\frac{\tau^+-\tau^-}{4}\right]^2,\frac{\tau^++\tau^-}2\right).
$$
Using $\tau_+=t+2\sqrt{-\ln|y|}$ and $\tau_-=t-2\sqrt{-\ln|y|}$, we obtain the particular double-wave solution
$$
u'(t,x,y)=(u_{1}(t,x,y),u_{2}(t,x,y))=(-\ln|y|,t), 
$$
which is indeed a double-wave solution of  (\ref{Eq:Example2}) obtained by a superposition of two single waves. In fact,
$$
d_p\tau^+=\lambda^+,\qquad d_p\tau^-=\lambda^-
$$
and
$$
[T_pu']=\left[ \begin{matrix}  0&0&-1/y\\ 1 &0& 0\end{matrix} \right]=\frac 12\gamma_+\otimes \lambda^++\frac 12 \gamma_-\otimes \lambda^- .
$$

\vskip .5cm
%\appendix
\noindent{\bf \Large Appendix (A modification of Frobenius theorem for integration).}
%\section{}
\vskip 0.5cm

Let us prove a result that seems to be missing in the literature despite its relation to the Frobenius Theorem and interest in the theory of hydrodynamic equations. Such a result was employed without proof in \cite{GL91} to study  the correspondence principle for $k$-waves. More specifically, it was stated in \cite{GL91} that the vector fields $\gamma_{(1)},\ldots,\gamma_{(r)}$ can be redefined to commute among themselves. Nevertheless, in the theory of  $k$-waves, the vector fields $\gamma_{(1)},\ldots,\gamma_{(r)}$ must satisfy the relation  (\ref{Rel:GammaDelta}). Hence, we are only interested in rescaling them by functions, which ensures that the rescaled vector fields still satisfy (\ref{Rel:GammaDelta}). Note that the Frobenius Theorem is not sufficient to prove the existence of an appropriate rescaling of  $\gamma_{(1)},\ldots,\gamma_{(r)}$ leading the latter to commute  among themselves. In fact,
the Frobenius Theorem proves that a family of vector fields $X_1,\ldots,X_r$ on a manifold $N$ satisfying $X_1\wedge\ldots\wedge X_r\neq 0$ and  spanning an involutive distribution of rank $k$ gives rise to an Abelian Lie algebra of vector fields  $\langle Y_1,\ldots,Y_r \rangle $ spanning the same distribution as $X_1,\ldots,X_r$ and such that
$$
Y_i=\sum_{j=1}^rf_{ij}X_j
$$
for a family of functions $f_{ij}\in C^\infty(N)$ with $i,j=1,\ldots,r$. Nevertheless, the generalised theory of $k$-waves demands that a simple rescaling of the vector fields $X_1,\ldots, X_r$ by a function must give rise to the family $Y_1,\ldots,Y_r$ of vector fields that commute among themselves. Meanwhile, the generalised theory of  $k$-waves also assumes a more restrictive form of the commutator between the vector fields  $\gamma_{(1)},\ldots,\gamma_{(r)}$. This restricted form is employed in the following theorem, which constitutes a modification of the Frobenius Theorem that is useful for our purposes. Note that, in the following theorem, for clarity, the Einstein convention is not used.

\begin{thm} {\bf (Modified Frobenius Theorem by rescaling)} Consider a family $X_1,\ldots,X_r$ of  vector fields on a $q$-dimensional manifold $N$ such that
	\begin{enumerate}
		\item  $X_1\wedge\ldots\wedge X_r\neq0$
		\item  $[X_i,X_j]=h_{i,j}^iX_i+h_{i,j}^jX_j$ 
	\end{enumerate} 
	for certain functions $h_{i,j}^k\in C^\infty(N)$ with $i,j,k=1,\ldots,r$. Then, there exist functions $f_1,\ldots,f_r\in C^\infty(N)$ such that $[f_iX_i,f_jX_j]=0$ for all $i,j=1,\ldots,r$. 
	
\end{thm}

\begin{proof}
	We proceed by induction. The result is immediate for $r=2$. Assume that our theorem is true for $r<q$ vector fields on $\mathbb{R}^q$ and let us prove that our theorem remains true for $r+1$ vector fields. Let   $X_1,\ldots,X_r,X_{r+1}$ be vector fields on $\mathbb{R}^{q}$ satisfying the conditions 1 and 2. The first $r$ vector fields can be rescaled, by the induction hypothesis, in order to obtain $Y_1,\ldots,Y_r$ so  that $[Y_i,Y_j]=0$ with $i,j=1,\ldots,r$. Then, there exist coordinates $\chi_1,\ldots,\chi_r,\eta_1,\ldots,\eta_s$ on $\mathbb{R}^q$ such that
	$$
	X_i=\frac{\partial}{\partial \chi_i},\qquad i=1,\ldots,r.
	$$
	Let us determine whether  there exists a function $f_{r+1}\in C^\infty(N)$ such that 
	\begin{equation}\label{Eq:FC}
		[X_i,f_{r+1}X_{r+1}]\wedge X_{i}=0,\qquad i=1,\ldots,r.
	\end{equation}
	Condition 2 and (\ref{Eq:FC}) imply that
	$$
	[X_i,f_{r+1}X_{r+1}]=(X_if_{r+1})X_{r+1}+f_{r+1}(h_{i,r+1}^{r+1}X_{r+1}+h_{i,r+1}^iX_i),\,\, i=1,\ldots,r,
	$$
	for certain functions $h_{i,r+1}^{r+1},h_{i,r+1}^i$ on $N$ with $i=1,\ldots,r$. Hence, to satisfy the conditions (\ref{Eq:FC}), the function $f_{r+1}$ must be such that 
	\begin{equation}\label{Eq:SC}
		X_i\ln|f_{r+1}|=-h_{i,r+1}^{r+1},\qquad i=1,\ldots,r.
	\end{equation}
	Let us prove that the above system of partial differential equations admits a solution. Its compatibility condition reads
	$$
	\frac{\partial h_{i,r+1}^{r+1}}{\partial \chi_j}=\frac{\partial h_{j,r+1}^{r+1}}{\partial \chi_i},\qquad i,j=1,\ldots,r.
	$$
	To show that the above compatibility condition holds, let us consider the Jacobi identities
	$$
	[X_i,[X_j,X_{r+1}]]+[X_j,[X_{r+1},X_{i}]]+[X_{r+1},[X_i,X_j]]=0,\qquad i,j=1,\ldots,r.
	$$
	Since $[X_i,X_j]=0$ by the induction hypothesis, the Jacobi identities reduces to
	$$
	[X_i,[X_j,X_{r+1}]]-[X_j,[X_{i},X_{r+1}]]=0,\qquad i,j=1,\ldots,r.
	$$
	Using assumption 2, we obtain that 
	$$
	[X_i,[X_j,X_{r+1}]]=\frac{\partial h_{j,r+1}^{r+1}}{\partial \chi_i}X_{r+1}+h_{j,r+1}^{r+1}[X_i,X_{r+1}]+\frac{\partial h^{j}_{j,r+1}}{\partial \chi_i}X_j
	$$
	for any $i,j=1,\ldots,r$, and, moreover,
	\begin{multline*}
		0=[X_j,[X_i,X_{r+1}]]-[X_i,[X_j,X_{r+1}]]=\left(\frac{\partial h_{i,r+1}^{r+1}}{\partial \chi_j}-\frac{\partial h_{j,r+1}^{r+1}}{\partial \chi_i}\right)X_{r+1}\\+h_{i,r+1}^{r+1}[X_j,X_{r+1}]+\frac{\partial h^i_{i,r+1}}{\partial \chi_j}X_i-h_{j,r+1}^{r+1}[X_i,X_{r+1}]-\frac{\partial h^j_{j,r+1}}{\partial \chi_i}X_j\\=
		\left(\frac{\partial h_{i,r+1}^{r+1}}{\partial \chi_j}-\frac{\partial h_{j,r+1}^{r+1}}{\partial \chi_i}\right)X_{r+1}-\left(h_{j,r+1}^{r+1}h_{i,r+1}^{i}-\frac{\partial h^i_{i,r+1}}{\partial \chi_j}\right)X_i\\+\left(h^{r+1}_{i,r+1}h_{j,r+1}^j-\frac{\partial h^j_{j,r+1}}{\partial \chi_i}\right)X_j.
	\end{multline*}
	Since $X_1\wedge\ldots\wedge X_{r+1}\neq 0$, it follows that
	\begin{equation}\label{eq:ConS}
		\frac{\partial h_{j,r+1}^{r+1}}{\partial \chi_i}=\frac{\partial h_{i,r+1}^{r+1}}{\partial \chi_j},\qquad i,j=1,\ldots,r,\qquad i\neq j.
	\end{equation}
	These equalities imply that the system (\ref{Eq:SC}) has a solution  $f_{r+1}$ and therefore, $X_{r+1}$ can be rescaled so that $\widetilde{X}_{r+1}=f_{r+1}X_{r+1}$ satisfies
	\begin{equation}\label{Eq:TC}
		[X_i,\widetilde{X}_{r+1}]\wedge X_{i}=0,\qquad i=1,\ldots,r.
	\end{equation}
	Let us show that it is possible to rescale $X_1,\ldots, X_r$ so that they still commute among themselves and, in addition,  commute with $X_{r+1}$. In fact,  $[f_iX_i,f_jX_j]=0$ for every $i,j=1,\ldots,r$ if and only if 
	$$
	f_i=f_i(\chi_i,\eta),\qquad i=1,\ldots,r.
	$$
	The conditions (\ref{Eq:TC}) for $\widetilde{X}_{r+1}$ imply that the latter takes the form
	$$
	\widetilde{X}_{r+1}=\sum_{i=1}^rf^i_{r+1}(\chi_i,\eta_1,\ldots,\eta_s)\frac{\partial}{\partial \chi_i}+\sum_{\mu=1}^sf^\mu_{r+1}(\eta_1,\ldots,\eta_s)\frac{\partial}{\partial \eta_\mu}
	$$
	for some functions of the form $f_{r+1}^i(\chi_i,\eta_1,\ldots,\eta_s)$ and $f^\mu_{r+1}(\eta_1,\ldots,\eta_s)$ for $i=1,\ldots,r$ and   $\mu=1,\ldots,s$. 
	If $[f_iX_i,\widetilde{X}_{r+1}]=0$, the function $f_i$ must satisfy the partial differential equation
	\begin{equation}\label{Eq:PDEFro}
		0=\widetilde {X}_{r+1}f_i-f_i\frac{\partial f^i_{r+1}}{\partial \chi_i}=f^i_{r+1}(\chi_i,\eta)\frac{\partial f_i}{\partial \chi_i}+\sum_{\mu=1}^sf^\mu_{r+1}(\eta)\frac{\partial f_i}{\partial \eta_\mu}-\frac{\partial f^i_{r+1}}{\partial \chi_i}f_i=0.
	\end{equation}
	For each $i=1,\ldots,r$, the equation (\ref{Eq:PDEFro}) is a linear PDE. Due to the dependence of its coefficients, it can be considered to be defined on $\mathbb{R}^{s+1}$ and it always admits  a solution by the method of characteristics. Hence, we obtain that
	$$
	[f_iX_i,f_jX_j]=0,\qquad i,j=1,\ldots,r+1.
	$$
	By the induction hypothesis, our theorem follows.
\end{proof}
\section{Acknowledgements}
A.M. Grundland was partially supported by an Operating Grant from NSERC (Canada) and the research grant ANR-11LABX-0056-LMHLabEX LMH (Fondation Math\'ematique Jacques Hadamard, France). J. de Lucas and A.M. Grundland acknowledge partial support from a HARMONIA  project (numer 2016/22/M/ST1/00542) funded by the National Science Centre (Poland). J. de Lucas would like very much to thank the staff of the Centre de Recherches Math\'ematiques (CRM) of the University of Montreal for their hospitality and assistance  during his stay at the CRM during the Covid pandemic. 

We would like to dedicate this paper to the memory of our friend Pavel Winternitz (Universit\'e de Montr\'eal), who passed away in February 2021. He  was always a source of inspiration and encouragement.

%\bibliographystyle{actapoly}
%\bibliography{biblio}

\end{document}